\documentclass[twoside,11pt,reqno]{amsart}
\usepackage{amsmath,amssymb,amscd,mathrsfs,amscd, todonotes,appendix}
\usepackage{graphics,verbatim}
\usepackage{todonotes}
\usepackage[shortlabels]{enumitem}
\usepackage{hyperref}
\usepackage{setspace} 

\usepackage{scalerel,stackengine}
\stackMath
\newcommand\reallywidehat[1]{%
\savestack{\tmpbox}{\stretchto{%
  \scaleto{%
    \scalerel*[\widthof{\ensuremath{#1}}]{\kern-.6pt\bigwedge\kern-.6pt}%
    {\rule[-\textheight/2]{1ex}{\textheight}}
  }{\textheight}%
}{0.5ex}}%
\stackon[1pt]{#1}{\tmpbox}%
}

\usepackage{graphicx} 
\usepackage{amsmath, amsfonts, amsthm, amssymb, mathrsfs}
\usepackage{tikz}
\usepackage{subcaption}
\usepackage{stackengine}
\usepackage{tikz-cd}
\usetikzlibrary{cd}
\usetikzlibrary{calc}
\usepackage{lipsum}
\usepackage{adjustbox}
\usepackage{cleveref}
\newcommand{\crefdefpart}[2]{%
  \hyperref[#2]{\namecref{#1}~\labelcref*{#1}~\ref*{#2}}%
}
\newcommand{\Crefdefpart}[2]{%
  \nameCref{#1}~\hyperref[#2]{\labelcref*{#1}~\ref*{#2}}%
}
\numberwithin{equation}{subsection}

\newcommand{\crly}[1]{\left\{#1\right\}}

\newcommand{\cur}[1]{\left(#1\right)}

\newcommand{\suml}[2]{\sum\limits_{#1}^{#2}}

\newcommand{\mf}[1]{\mathfrak{#1}}
\newcommand{\mc}[1]{\mathcal{#1}}
\newcommand{\del}{\partial}
\newcommand{\Ker}{\textup{Ker}}

\newcommand{\IM}{\textup{Im}}
\newcommand{\id}{\textup{id}}
\newcommand{\Hom}[2]{\textup{Hom}_{#1}\cur{#2}}
\newcommand{\ind}[2]{\textup{ind}_{#1}^{#2}}
\newcommand{\res}[2]{\textup{res}_{#1}^{#2}}
\newcommand{\Rind}[3]{\textup{R}^{#1}\textup{ind}_{#2}^{#3}}
\newcommand{\rRind}[4]{\textup{R}_{#1}^{#2}\textup{ind}_{#3}^{#4}}

\newcommand{\Ext}[2]{\textup{Ext}_{#1}^{#2}}

\newcommand{\Tot}{\textup{Tot}}
\newcommand{\Mod}[1]{\textup{Mod}\cur{#1}}

\makeatletter
\newcommand*{\bigcdot}{\mathpalette\bigcdot@{0.75}}
\newcommand*\bigcdot@[2]{\mathbin{\vcenter{\hbox{\scalebox{#2}{$\m@th#1\bullet$}}}}}
\makeatother

\newtheorem{theorem}{Theorem}[subsection]
\newtheorem{lemma}[theorem]{Lemma}
\newtheorem{proposition}[theorem]{Proposition}
\newtheorem{corollary}[theorem]{Corollary}

\newtheorem{remark}[theorem]{Remark}
\newtheorem{observation}[theorem]{Observation}

\newcommand{\sref}[1]{\textup{sequence}~(\ref{#1})}

\usepackage[backend=bibtex, style=alphabetic, sorting=nyt, maxnames=99]{biblatex}
\DeclareFieldFormat{postnote}{#1}
\DeclareFieldFormat{multipostnote}{#1}

\begin{document}

\title{On the Relative Cohomology for Algebraic Groups}

\author{\sc Gabriel T. Loos}
\address
{Department of Mathematics\\ University of Georgia \\
Athens\\ GA~30602, USA}
\email{gabriel.loos@uga.edu}

\begin{abstract}
Let $G$ be an algebraic group over a field $k$, and $M$ and $N$ be $G$-modules. In 1961, Hochschild showed how one can define the cohomology groups $\Ext{G}{i}(M,N)$. Kimura, in 1965, showed that one can generalize this to get relative cohomology for algebraic groups. The original cohomology groups play an important role in understanding the representation theory of $G$, but the role of relative cohomology is still not well understood. 

In this paper the author expands upon the work of Kimura to prove foundational results about the relative cohomology. The author starts by giving the definitions of relative exact sequences and relative injective modules and proves a variety of basic properties for each that will be essential to define relative cohomology and obtain a relative Grothendieck spectral sequence. In particular, the induction functor will play an important role when studying the relative injective modules. Once the necessary ground work is laid, the definition of relative cohomology is given. Finally, it is stated when there is a relative Grothendieck spectral sequence, and many consequences and examples are provided.
\end{abstract}

\maketitle

\section{Introduction}

\subsection{}

Let $G$ be an algebraic group over a field $k$, and $M$ be a $G$-module. The cohomology groups, $\textup{H}^{i}\cur{G,M}$, play an important role in understanding the representation theory of $G$. For an extensive overview of the many applications of the $G$-cohomology, we refer the reader to \cite[Part II]{Jan03}. Namely, the $G$-cohomology is crucial when studying highest weight modules. In particular, one famous result, which can be found in \cite[II.4.13]{Jan03}, gives that
\begin{equation}\label{eoowaim}
    \Ext{G}{i}\cur{\Delta(\lambda),\nabla(\mu)}\cong
\begin{cases}
    k & \textup{ if $i=0$ and $\lambda=\mu$}\\
    0 & \textup{ otherwise}
\end{cases},
\end{equation}
where $\nabla(\mu)=\ind{B}{G}(\mu)$ and $\Delta(\lambda)=\nabla(-w_0\lambda)^*$.
\Cref{eoowaim} gives rise to cohomological conditions for when a $G$-module, $M$, has a good-filtration (respectively Weyl-filtration), i.e. when there exists a chain of submodules of $M$,
$$0=M_0\subsetneq M_1\subsetneq M_2\subsetneq\cdots,$$ such that $M=\bigcup\limits_{i\geq0}M_i$ and for $i\geq1$, $M_{i}/M_{i-1}\cong\nabla(\mu_i)$ (respectively $M_{i}/M_{i-1}\cong\Delta(\lambda_i)$) for some $\mu_i\in X_+(T)$ (respectively $\lambda_i\in X_+(T)$).

%
%

In Category $\mathcal{O}$, relative cohomology has been studied extensively and found to have many applications. Boe, Kujawa, and Nakano, in \cite{BKN10}, use relative cohomology to establish a connection between the support varieties of $\mf{g}$ and the detecting subalgebras (which they introduce), where $\mf{g}$ is a simple classical Lie superalgebra. Later, D. Grantcharov, N. Grantcharov, Nakano, and Wu, in \cite{GGNW21}, introduce a special class of parabolic subalgebras of $\mf{g}$, which they use along with relative cohomology to prove a conjecture from \cite{BKN10} showing that certain support varieties are isomorphic. Most recently, Lai, Nakano, and Wilbert, in \cite{LNW25}, develop a more general theory of Category $\mathcal{O}$ for a quasi-reductive Lie superalgebra $\mf{g}$, and use relative cohomology to explicitly realize the cohomology ring and to establish finite generation results for the parabolic Category $\mathcal{O}$.

Kimura, in \cite{Kim65}, showed how to define relative cohomology for an algebraic group, which we include the definition of in \Cref{relative cohomology} on page \pageref{relative cohomology}. In order to define relative cohomology, one must first define relative exact sequences, defined in \Cref{relative exact sequence} on page \pageref{relative exact sequence}, and  relative injective modules, defined in \Cref{relative injective module} on page \pageref{relative injective module}, along with showing that $\Mod{G}$, the category of rational $G$ modules, has enough relative injectives, all of which Kimura did in \cite{Kim65}. Kimura then proves some initial results about relative cohomology in the characteristic zero case.
In \cite{Kim66}, Kimura shows that there exist long exact sequences for relative $\Ext{}{}$, along with showing that relative $\Ext{}{n}$ is equivalent to the set of equivalence classes of ``$n$-fold relative extensions''. Later in the paper, he gives criterion for when a relative injective $G$-module will also be a relative injective $H$-module for some closed subgroup $H\leq G$. 

It should be noted that some of Kimura's results require characteristic zero, but through out this paper our methods will work for arbitrary characteristic. 

\subsection{}

In this paper we will study relative cohomology from a more functorial viewpoint by using the induction functor to our advantage. This perspective will allow us to reformulate the results of Kimura which are only for characteristic zero, as results for arbitrary characteristic, although we will sometimes need stronger conditions on the subgroups of $G$. We will also build up a more general theory of relative right derived functors from $\Mod{G}$ to $\Mod{G'}$, where $G$ and $G'$ are algebraic groups and $\Mod{G}$ (respectively $\Mod{G'}$) is the category of rational $G$-modules (respectively $G'$-modules). In particular, we will give sufficient conditions for when there exists a relative Grothendieck spectral sequence, and then we will apply the spectral sequence to get relative versions of generalized Frobenius reciprocity, Shapiro's lemma, and generalized tensor identity.

\subsection{}

The paper is organized as follows. 
%
%
In \Cref{Definitions and Basic Properties}, we give the basic definitions along with some preliminary results.
%
%
In the following section (\Cref{Induction}) we characterize all relative injective modules using the induction functor. Afterwards we give conditions for when induction is exact on relative exact sequences.
%

In \Cref{Relative Right Derived Functor}, we introduce relative right derived functors, and prove some preliminary results which will be analogous to the already well known statements for the normal right derived functor.
%
%
%
%
In the following section (\Cref{Cohomology}) we give sufficient conditions for when there is a relative Grothendieck spectral sequence. 
%
%
Afterwards we study the relative cohomology of factor groups in \Cref{Factor Groups}, where we state a isomorphism which is motivated by the Lyndon-Hochschild-Serre spectral sequence.
%

We end by computing several examples in \Cref{Examples}. In particular, we show 
\begin{enumerate}
    \item that in characteristic zero, relative cohomology is trivial for finite dimensional modules, 
    \item how relative cohomology will behave with parabolic subgroups, 
    \item that there is a relative Ext-transfer theorem, 
    \item how relative cohomology interacts with CPS coupled parabolic systems, and
    \item that relative cohomology is most interesting when looking relative to a Levi subgroup.
\end{enumerate}

\subsection{Acknowledgements}

The author would like to thank his Ph.D advisor, Daniel Nakano, for many helpful discussions and comments pertaining to the contents of this paper. He would also like to thank Christopher Bendel and the referees for useful comments and suggestions on an earlier version of this manuscript.
\section{Definitions and Basic Properties}\label{Definitions and Basic Properties}

\subsection{Relative Exact Sequences}\label{relative exact sequence}

The definition of a ``$(G,H)$-exact sequence'' was first given by Kimura in \cite[pp. 271-272]{Kim65}. For convenience, we will include the definition here as well.

Let $G$ be an algebraic group over a field $k$, and $H$ a closed subgroup of $G$. An exact sequence of $G$-modules,
\[
\adjustbox{scale=1}{%
\begin{tikzcd}
    \cdots\arrow[r] & M_{i-1}\arrow[r, "d_{i-1}"] & M_{i}\arrow[r, "d_{i}"] & M_{i+1}\arrow[r] & \cdots, 
\end{tikzcd}
}
\]
is $(G,H)$-exact if $\Ker\cur{d_{i}}$ is a direct $H$-module summand of $M_{i}$ for all $i$. Equivalently, the sequence is $(G,H)$-exact if for all $i$, there exists an $H$-module homomorphism $t_{i}:M_{i}\to M_{i-1}$ such that $d_{i-1}t_{i}+t_{i+1}d_{i}=\id_{M_{i}}$. Note that any exact sequence of $G$-modules will be $(G,\{1\})$-exact.

From the definition, it is easy to see that for $H\leq H'\leq G$, if a sequence is $(G,H')$-exact, then it is also $(G,H)$-exact.

\begin{lemma}\label{resd}
    Let $H$ and $H'$ be closed subgroups of $G$ such that $H\leq H'\leq G$, if 
    \[
    \adjustbox{scale=1}{%
    \begin{tikzcd}
        \cdots\arrow[r] & M_{i-1}\arrow[r, "d_{i-1}"] & M_{i}\arrow[r, "d_{i}"] & M_{i+1}\arrow[r] & \cdots 
    \end{tikzcd}
    }
    \]
    is $(G,H')$-exact, then it is also $(G,H)$-exact.
\end{lemma}

\begin{proof}
    For all $i$, there exists an $H'$-module homomorphism $t_{i}:M_{i}\to M_{i-1}$ such that $d_{i-1}t_{i}+t_{i+1}d_{i}=\id_{M_{i}}$. Clearly, since $H\leq H'$, each $t_{i}$ is also an $H$-module homomorphism, and hence the sequence is also $(G,H)$-exact.
\end{proof}
\subsection{}

The following two lemmas, while not being very difficult, will be important facts about relative exact sequences when proving later results, so we go ahead and prove them now. In general, we will be interested in how functors interact with relative exact sequences. Getting a handle on this interaction will typically be a challenge, but for the functors $-\otimes V$ and $\Hom{H}{E,-}$ from $\Mod{G}$ to $\Mod{G}$ (resp. $\Mod{k}$) we have a better understanding.

\begin{lemma}\label{tetretre}
    If $-\otimes V:\Mod{G}\to\Mod{G}$ is exact, then it takes $(G,H)$-exact sequences to $(G,H)$-exact sequences.
\end{lemma}

\begin{proof}
    Consider a $(G,H)$-exact sequence
    \[
    \adjustbox{scale=1}{%
    \begin{tikzcd}
        \cdots\arrow[r] & M_{i}\arrow[r, "d_{i}"] & M_{i+1}\arrow[r] & \cdots.
    \end{tikzcd}
    }
    \]
    As an $H$-module, $M_{i}\cong\Ker\cur{d_{i}}\oplus\IM\cur{d_{i}}$. So (as an $H$-module) $M_{i}\otimes V\cong\cur{\Ker\cur{d_{i}}\oplus\IM\cur{d_{i}}}\otimes V\cong\cur{\Ker\cur{d_{i}}\otimes V}\oplus\cur{\IM\cur{d_{i}}\otimes V}$. Hence, $-\otimes V$ sends $(G,H)$-exact sequences to $(G,H)$-exact sequences.
\end{proof}

\begin{lemma}\label{hwsgoreie}
    Let $H$ and $H'$ be closed subgroups of $G$ such that $H\leq H'\leq G$, $E$ be an $H$-module, and
    \[
    \adjustbox{scale=1}{%
    \begin{tikzcd}
        0\arrow[r] & M_{1}\arrow[r, "f"] & M_{2}\arrow[r, "g"] & M_{3}\arrow[r] & 0 
    \end{tikzcd}
    }
    \]
    be $(G,H')$-exact. Then
    \[
    \adjustbox{scale=1}{%
    \begin{tikzcd}
        0\arrow[r] & \Hom{H}{E,M_{1}}\arrow[r, "f\circ-"] & \Hom{H}{E,M_{2}}\arrow[r, "g\circ-"] & \Hom{H}{E,M_{3}}\arrow[r] & 0 
    \end{tikzcd}
    }
    \]
    is exact.
\end{lemma}

\begin{proof}
    $\Hom{H}{E,-}$ is left exact, so for exactness we just need to show that $g\circ-$ is surjective. Let $\alpha\in\Hom{H}{E,M_3}$. Since
    \[
    \adjustbox{scale=1}{%
    \begin{tikzcd}
        0\arrow[r] & M_{1}\arrow[r, "f"] & M_{2}\arrow[r, "g"] & M_{3}\arrow[r] & 0 
    \end{tikzcd}
    }
    \]
    is $(G,H')$-exact, there is an $H'$-homomorphism $h:M_3\to M_2$ such that $g\circ h=\id_{M_3}$. Hence, $h\circ\alpha\in\Hom{H}{E,M_2}$ such that $g\circ h\circ\alpha=\alpha\in\Hom{H}{E,M_3}$. Therefore,
    \[
    \adjustbox{scale=1}{%
    \begin{tikzcd}
        0\arrow[r] & \Hom{H}{E,M_{1}}\arrow[r, "f\circ-"] & \Hom{H}{E,M_{2}}\arrow[r, "g\circ-"] & \Hom{H}{E,M_{3}}\arrow[r] & 0 
    \end{tikzcd}
    }
    \]
    is exact.
\end{proof}
\subsection{Relative Injective Modules}\label{relative injective module}

The definition of a ``$(G,H)$-injective module'' was first given by Kimura in \cite[p. 272]{Kim65}. For convenience, we will include the definition here as well.

A $G$-module $I$ is called $(G,H)$-injective if for all $(G,H)$-exact sequences 
\[
\adjustbox{scale=1}{%
\begin{tikzcd}
    0\arrow[r] & M_{1}\arrow[r, "d_{1}"] & M_{2}\arrow[r, "d_{2}"] & M_{3}\arrow[r] & 0 
\end{tikzcd}
}
\]
and any $G$-module homomorphism $f:M_{1}\to I$, there exists a $G$-module homomorphism $h:M_{2}\to I$ such that $f=hd_{1}$, i.e. making the following diagram commute:
\[
\adjustbox{scale=1}{%
\begin{tikzcd}
    0\arrow[r] & M_{1}\arrow[r, "d_{1}"]\arrow[d,"f", swap] & M_{2}\arrow[r, "d_{2}"]\arrow[dl, "h"] & M_{3}\arrow[r] & 0 \\
     & I &  &  & 
\end{tikzcd}
}.
\]
We point out that by choosing $H$ to be the identity subgroup, we simply recover the definition of injective modules, thus this gives a generalization. One can also clearly see that if $I$ is injective, then $I$ is $(G,H)$-injective. In fact, one has the following:

\begin{lemma}\label{rimu}
    Let $M$ be a $G$-module, and $H$ and $H'$ be closed subgroups of $G$ such that $H\leq H'\leq G$. If $M$ is $(G,H)$-injective, then $M$ is $(G,H')$-injective.
\end{lemma}

\begin{proof}
    Consider a $(G,H')$-exact sequence
    \begin{equation}\label{eq:3}
        \adjustbox{scale=1}{%
        \begin{tikzcd}
            0\arrow[r] & M_{1}\arrow[r, "d_{1}"] & M_{2}\arrow[r, "d_{2}"] & M_{3}\arrow[r] & 0 
        \end{tikzcd}
        }
    \end{equation}
    and any $G$-module homomorphism $f:M_{1}\to M$. Now, from \Cref{resd}, \sref{eq:3} is also a $(G,H)$-exact sequence, and so there exists a $G$-module homomorphism $h:M_{2}\to M$ such that $f=hd_{1}$. Hence, $M$ is $(G,H')$-injective.
\end{proof}
\subsection{}

One easy example that can be worked out is when $H=G$. Let $M$ be a $G$-module, and consider a $(G,G)$-exact sequence, i.e. a split sequence
\begin{equation}\label{eq:4}
    \adjustbox{scale=1}{%
    \begin{tikzcd}
        0\arrow[r] & M_{1}\arrow[r, "d_{1}"] & M_{2}\arrow[r, "d_{2}"] & M_{3}\arrow[r]& 0 
    \end{tikzcd}
    }
\end{equation}
and any map $f:M_1\to M$. Since \sref{eq:4} is split, there is a $G$-module homomorphism $\pi:M_2\to M_1$ such that $\pi\circ d_1=\id_{M_1}$. Now consider $h=f\circ\pi:M_2\to M$, then $h\circ d_1=f\circ\pi\circ d_1=f$, and so $M$ is $(G,G)$-injective. Hence, every $G$-module $M$ is $(G,G)$-injective.
\subsection{}

The following two results are fairly basic but will be important facts for proving later results so we include them now.

\begin{lemma}\label{srses}\textup{\cite[p. 90]{Kim66}}
    Let
    \[
    \adjustbox{scale=1}{%
    \begin{tikzcd}
        0\arrow[r] & M_{1}\arrow[r, "d_{1}"] & M_{2}\arrow[r, "d_{2}"] & M_{3}\arrow[r] & 0 
    \end{tikzcd}
    }
    \]
    be a $(G,H)$-exact sequence. If $M_1$ is $(G,H)$-injective, then the sequence splits.
\end{lemma}

\begin{proof}
    Consider the identity map $\id_{M_{1}}:M_{1}\to M_{1}$, then since $M_1$ is $(G,H)$-injective, there is a $G$-homomorphism $h:M_{2}\to M_{1}$ such that $\id_{M_{1}}=h\circ d_{2}$, and so the sequence splits.
\end{proof}

The next lemma will be important for proving that there is a relative cohomological criterion for a module to be relative injective.

\begin{lemma}\label{hire}
    Let
    \[
    \adjustbox{scale=1}{%
    \begin{tikzcd}
        0\arrow[r] & M_{1}\arrow[r, "d_{1}"] & M_{2}\arrow[r, "d_{2}"] & M_{3}\arrow[r] & 0 
    \end{tikzcd}
    }
    \]
    be a $(G,H)$-exact sequence and $I$ a $(G,H)$-injective module, then the sequence
    \[
    \adjustbox{scale=1}{%
    \begin{tikzcd}
        0\arrow[r] & \Hom{G}{M_{3},I}\arrow[r, "-\circ d_{2}"] & \Hom{G}{M_{2},I}\arrow[r, "-\circ d_{1}"] & \Hom{G}{M_{1},I}\arrow[r] & 0 
    \end{tikzcd}
    }
    \]
    is exact.
\end{lemma}

\begin{proof}
    $\Hom{G}{-,I}$ is left exact, so we just need to show that $-\circ d_{1}$ is surjective. Consider a $G$-homomorphism $f:M_1\to I$, since $I$ is $(G,H)$-injective, there is a $G$-homomorphism $h:M_2\to I$ such that $f=h\circ d_1$. Hence, $-\circ d_{1}$ is surjective.
\end{proof}
\subsection{Direct Sums and Summands of Relative Injective Modules}

When studying the normal theory, it is known that the direct sum of two injective modules will also be injective, and that if a module is the direct summand of an injective module, then that module will also be injective. And in fact, the same will hold true when studying relative injective modules, with the generalization of the later fact being important for the overarching theory.

\begin{lemma}\label{dsri}
    Let $I$ and $J$ be $(G,H)$-injective modules, then $I\oplus J$ is $(G,H)$-injective.
\end{lemma}

\begin{proof}
    Consider a $(G,H)$-exact sequence
    \[
    \adjustbox{scale=1}{%
    \begin{tikzcd}
        0\arrow[r] & M_{1}\arrow[r, "d_{1}"] & M_{2}\arrow[r, "d_{2}"] & M_{3}\arrow[r] & 0 
    \end{tikzcd}
    }
    \]
    and any map $f:M_{1}\to I\oplus J$. Then one also has the maps $\pi_{I}\circ f:M_{1}\to I$ and $\pi_{J}\circ f:M_{1}\to J$. Since both $I$ and $J$ are relative $(G,H)$-injective, there exists maps $h_{I}:M_{2}\to I$ and $h_{J}:M_{2}\to J$ such that $\pi_{I}\circ f=h_{I}\circ d_{1}$ and $\pi_{J}\circ f=h_{J}\circ d_{1}$, i.e. there are commutative diagrams
    \[
    \adjustbox{scale=1}{%
    \begin{tikzcd}
        0\arrow[r] & M_{1}\arrow[r, "d_{1}"]\arrow[d,"f", swap] & M_{2}\arrow[r, "d_{2}"]\arrow[ddl, "h_{I}"] &  M_{3}\arrow[r] & 0 \\
         & I\oplus J\arrow[d, "\pi_{I}", swap] &  &  & \\
         & I &  &  & 
    \end{tikzcd}
    }
    \]
    and
    \[
    \adjustbox{scale=1}{%
    \begin{tikzcd}
        0\arrow[r] & M_{1}\arrow[r, "d_{1}"]\arrow[d,"f", swap] & M_{2}\arrow[r, "d_{2}"]\arrow[ddl, "h_{J}"] &  M_{3}\arrow[r] & 0 \\
         & I\oplus J\arrow[d, "\pi_{J}", swap] &  &  & \\
         & J &  &  & 
    \end{tikzcd}
    }.
    \]
    One can then see that $i_{I}\circ \pi_{I}\circ f=i_{I}\circ h_{I}\circ d_{1}$ and $i_{J}\circ \pi_{J}\circ f=i_{J}\circ h_{J}\circ d_{1}$, and so $\cur{i_{I}\circ h_{I}+i_{J}\circ h_{J}}\circ d_{1}=i_{I}\circ h_{I}\circ d_{1}+i_{J}\circ h_{J}\circ d_{1}=i_{I}\circ \pi_{I}\circ f+i_{J}\circ \pi_{J}\circ f=\cur{i_{I}\circ \pi_{I}+i_{J}\circ \pi_{J}}\circ f=f$. Therefore, $I\oplus J$ is relative $(G,H)$-injective.
\end{proof}

\begin{lemma}\label{dsim}
    Let $M$ be a $G$-module. If $M$ is a direct summand of a $(G,H)$-injective module $I$, then $M$ is $(G,H)$-injective.
\end{lemma}

\begin{proof}
    Consider a $(G,H)$-exact sequence
    \[
    \adjustbox{scale=1}{%
    \begin{tikzcd}
        0\arrow[r] & M_{1}\arrow[r, "d_{1}"] & M_{2}\arrow[r, "d_{2}"] & M_{3}\arrow[r] & 0,
    \end{tikzcd}
    }
    \]
    and any $G$-module homomorphism $f:M_1\to M$. We want to find a homomorphism $h:M_2\to M$ such that $h\circ d_1=f$.
    
    Since $M$ is a direct summand of $I$, there is an injective homomorphisms $i:M\to I$ and a surjective homomorphism $\pi:I\to M$ such that $\pi\circ i=\id_{M}$. Consider the homomorphism $i\circ f:M_1\to I$, since $I$ is $(G,H)$-injective, there exists a homomorphism $\tilde h:M_2\to I$ such that $\tilde h\circ d_1=i\circ f$, i.e. the following diagram commutes
    \[
    \adjustbox{scale=1}{%
    \begin{tikzcd}
        0\arrow[r] & M_{1}\arrow[r, "d_{1}"]\arrow[d, "f", swap] & M_{2}\arrow[r, "d_{2}"]\arrow[ddl, "\tilde h"] & M_{3}\arrow[r] & 0 \\
         & M\arrow[d, "i", swap, shift right=0.5ex] &  &  &  \\
         & I\arrow[u, "\pi", swap, shift right=0.5ex] &  &  &  
    \end{tikzcd}
    }.
    \]
    Hence, $\pi\circ\tilde h\circ d_1=\pi\circ i\circ f=f$. So letting $h=\pi\circ\tilde h:M_2\to M$ gives the desired homomorphism, thus $M$ is $(G,H)$-injective.
\end{proof}
\section{Induction}\label{Induction}

\subsection{}

Let $H$ be a closed subgroup of $G$, then there are functors $$\res{H}{G}:\Mod{G}\to\Mod{H}$$ and $$\ind{H}{G}:\Mod{H}\to\Mod{G}.$$ $\res{H}{G}$ is defined in the obvious way of restricting a $G$-module to an $H$-module, so is clearly exact. $\ind{H}{G}$ can be defined in one of the following equivalent ways:
\begin{align*}
    \ind{H}{G}(M)=&\cur{M\otimes k[G]}^H\\
    =&\Hom{H}{k,M\otimes k[G]}\\
    =&\{f:G\to M|f(gh)=h^{-1}f(g)\textup{ for all }g\in G\textup{ and all }h\in H\},
\end{align*}
where $k[G]$ is the algebra of regular functions on $G$.
From this it is clear that $\ind{H}{G}$ is left exact. For a detailed description of $\ind{H}{G}$ and its properties, the reader is referred to \cite[I.3]{Jan03}.

In this section, we will use the induction functor to study relative injective modules. Note that in Kimura's paper, \cite{Kim65}, he only uses $-\otimes k[G]^H:\Mod{G}\to\Mod{G}$ to study relative injectives. By using $\ind{H}{G}:\Mod{H}\to\Mod{G}$, we will be able to produce stronger results that will encapsulate those of Kimura.
\subsection{Induction of Relative Injective Modules}

We are now ready to prove what will be some very crucial results. The first being that induction will take relative injective modules to relative injective modules.

\begin{theorem}\label{iriri}
    Let $H$ and $K$ be closed subgroups of $G$, and $M$ be an $H$-module such that $M$ is $(H,H\cap K)$-injective, then $\ind{H}{G}(M)$ is $(G,K)$-injective.
\end{theorem}

\begin{proof}
    Consider a $(G,K)$-exact sequence
    \begin{equation}\label{eq:5}
        \adjustbox{scale=1}{%
        \begin{tikzcd}
            0\arrow[r] & M_{1}\arrow[r, "d_{1}"] & M_{2}\arrow[r, "d_{2}"] & M_{3}\arrow[r] & 0 
        \end{tikzcd}
        }
    \end{equation}
    and consider any $G$-module homomorphism $f:M_1\to \ind{H}{G}(M)$. Since $\ind{H}{G}$ is right adjoint to $\res{H}{G}$, there exists a unique $i_{M_{1}}:M_1\to\ind{H}{G}\cur{\res{H}{G}\cur{M_1}}$ and some $\tilde f:\res{H}{G}(M_1)\to M$, such that the following diagram commutes: 
    \[
    \adjustbox{scale=1}{%
    \begin{tikzcd}
        M_1\arrow[r, "i_{M_{1}}"]\arrow[dr, "f", swap] & \ind{H}{G}\cur{\res{H}{G}\cur{M_1}}\arrow[d, "\ind{H}{G}\cur{\tilde f}"] \\
         & \ind{H}{G}\cur{M}
    \end{tikzcd}
    }.
    \]
    Now, apply $\res{H}{G}$ to \sref{eq:5} to get
    \begin{equation}\label{eq:6}
        \adjustbox{scale=1}{%
        \begin{tikzcd}
            0\arrow[r] & \res{H}{G}\cur{M_{1}}\arrow[r, "\res{H}{G}\cur{d_{1}}"]\arrow[d, "\tilde f", swap] & \res{H}{G}\cur{M_{2}}\arrow[r, "\res{H}{G}\cur{d_{2}}"] & \res{H}{G}\cur{M_{3}}\arrow[r] & 0 \\
             & M &  &  & 
        \end{tikzcd}
        }.
    \end{equation}
    By our assumption, \sref{eq:5} is split as a $K$-module, and by \Cref{resd} is also split as an $H\cap K$-module, so \sref{eq:6} is $(H,H\cap K)$-exact. Therefore, by the assumption that $M$ is $(H,H\cap K)$-injective, there exists an $H$-module homomorphism $\tilde h:\res{H}{G}\cur{M_{2}}\to M$ such that $\tilde h\circ\res{H}{G}\cur{d_{1}}=\tilde f$. Now, apply $\ind{H}{G}$ to \sref{eq:6} to get
    \[
    \adjustbox{scale=1}{%
    \begin{tikzcd}
        0\arrow[r] & \ind{H}{G}\cur{\res{H}{G}\cur{M_{1}}}\arrow[rr, "\ind{H}{G}\cur{\res{H}{G}\cur{d_{1}}}"]\arrow[d, "\ind{H}{G}\cur{\tilde f}", swap] &  & \ind{H}{G}\cur{\res{H}{G}\cur{M_{2}}}\arrow[rr, "\ind{H}{G}\cur{\res{H}{G}\cur{d_{2}}}"]\arrow[dll, "\ind{H}{G}\cur{\tilde h}"] &  & \ind{H}{G}\cur{\res{H}{G}\cur{M_{3}}} \\
         & \ind{H}{G}\cur{M} &  &  &  & 
    \end{tikzcd}
    }.
    \]
    Next, since $\ind{H}{G}$ is right adjoint to $\res{H}{G}$, the following diagram commutes
    \[
    \adjustbox{scale=0.92}{%
    \begin{tikzcd}
        0\arrow[r] & M_{1}\arrow[rr, "d_{1}"]\arrow[d, "i_{M_{1}}"] &  & M_{2}\arrow[rr, "d_{2}"]\arrow[d, "i_{M_{2}}"] &  & M_{3}\arrow[r]\arrow[d, "i_{M_{3}}"] & 0 \\
        0\arrow[r] & \ind{H}{G}\cur{\res{H}{G}\cur{M_{1}}}\arrow[rr, "\ind{H}{G}\cur{\res{H}{G}\cur{d_{1}}}"]\arrow[d, "\ind{H}{G}\cur{\tilde f}", swap] &  & \ind{H}{G}\cur{\res{H}{G}\cur{M_{2}}}\arrow[rr, "\ind{H}{G}\cur{\res{H}{G}\cur{d_{2}}}"]\arrow[dll, "\ind{H}{G}\cur{\tilde h}"] &  & \ind{H}{G}\cur{\res{H}{G}\cur{M_{3}}} \\
         & \ind{H}{G}\cur{M} &  &  &  & 
    \end{tikzcd}
    }.
    \]
    So, $f=\ind{H}{G}\cur{\tilde f}\circ i_{M_1}=\ind{H}{G}\cur{\tilde h}\circ\ind{H}{G}\cur{\res{H}{G}\cur{d_{1}}}\circ i_{M_1}=\ind{H}{G}\cur{\tilde h}\circ i_{M_2}\circ d_1$. Let $h=\ind{H}{G}\cur{\tilde h}\circ i_{M_2}:M_2\to\ind{H}{G}(M)$, hence we have found a $G$-module homomorphism, $h$, such that $f=h\circ d_1$. Therefore, $\ind{H}{G}(M)$ is $(G,K)$-injective.
\end{proof}

Applying \Cref{iriri} to the case $K=H$ we get the following:

\begin{corollary}\label{iri}
    Let $M$ be an $H$-module, then $\ind{H}{G}(M)$ is $(G,H)$-injective.
\end{corollary}

We remark that in \textup{\cite[Proposition 2.1]{Kim65}}, it is shown that if $M$ is a $G$-module, then $\ind{H}{G}(M)\cong M\otimes k[G]^H$ is $(G,H)$-injective.

\begin{corollary}\label{eim}\textup{\cite[Proposition 2.2]{Kim65}}
    Every $G$-module $M$ can be embedded into a $(G,H)$-injective module such that it is a direct summand as an $H$-module.
\end{corollary}

\begin{proof}
    Let $M$ be a $G$-module and consider $\id_{\res{H}{G}(M)}:\res{H}{G}(M)\to \res{H}{G}(M)$. Since $\ind{H}{G}$ is right adjoint to $\res{H}{G}$, there exists a unique $\pi_{\res{H}{G}(M)}:\res{H}{G}\cur{\ind{H}{G}\cur{\res{H}{G}(M)}}\to\res{H}{G}(M)$ and some $i_{M}:M\to\ind{H}{G}\cur{\res{H}{G}(M)}$, such that the following diagram commutes: 
    \[
    \adjustbox{scale=1}{%
    \begin{tikzcd}
        \res{H}{G}\cur{\ind{H}{G}\cur{\res{H}{G}(M)}}\arrow[r, "\pi_{\res{H}{G}(M)}"] & \res{H}{G}(M) \\
         & \res{H}{G}(M)\arrow[ul, "\res{H}{G}(i_{M})"]\arrow[u, "\id_{\res{H}{G}(M)}", swap]
    \end{tikzcd}
    }.
    \]
    So it is clear that $M$ is a direct summand as an $H$-module of $\ind{H}{G}(M)$, which is $(G,H)$-injective by \Cref{iri}.
\end{proof}
\subsection{}

Using the above results, we can now characterize all relative injective modules in terms of induction.

\begin{lemma}\label{riiffds}
    A $G$-module $M$ is $(G,H)$-injective if and only if $M$ is a direct summand of some $\ind{H}{G}(N)$.
\end{lemma}

\begin{proof}
    Assume $M$ is $(G,H)$-injective, by \Cref{eim} the sequence
    \[
    \adjustbox{scale=1}{%
    \begin{tikzcd}
        0\arrow[r] & M\arrow[r, "i"] & \ind{H}{G}(M)\arrow[r, "p"] & \ind{H}{G}(M)/\IM(i)\arrow[r] & 0 
    \end{tikzcd}
    }
    \]
    is $(G,H)$-exact. It then follows by \Cref{srses} that $M$ is a direct summand of $\ind{H}{G}(M)$.

    Now assume that $M$ is a direct summand of some $\ind{H}{G}(N)$. Then, by \Cref{iri} and \Cref{dsim}, $M$ is $(G,H)$-injective.
\end{proof}

\begin{lemma}\label{triri}\textup{\cite[Proposition 2.1]{Kim65}}
    Let $M$ and $I$ be $G$-modules such that $I$ is $(G,H)$-injective, then $I\otimes M$ is $(G,H)$-injective.
\end{lemma}

\begin{proof}
    $I$ is a direct summand of $\ind{H}{G}(I)$, so $I\otimes M$ is a direct summand of $\ind{H}{G}(I)\otimes M\cong\ind{H}{G}\cur{I\otimes\res{H}{G}(M)}$. So by \Cref{iri} and \Cref{dsim} $I\otimes M$ is $(G,H)$-injective.
\end{proof}
\subsection{Relative Injective Resolutions}\label{rird}

The definition of a ``$(G,H)$-injective resolution'' was fist given by Kimura in \cite[p. 273]{Kim65}. For convenience, we will include the definition here as well.

For a $G$-module $M$, a $(G,H)$-injective resolution is a $(G,H)$-exact sequence
\[
\adjustbox{scale=1}{%
\begin{tikzcd}
    0\arrow[r] & M\arrow[r] & I_{0}\arrow[r] & I_{1}\arrow[r] & \cdots 
\end{tikzcd}
}
\]
such that each $I_{j}$ is $(G,H)$-injective.

\begin{lemma}\textup{\cite[pp. 273-274]{Kim65}}
    Every $G$-module has a $(G,H)$-injective resolution.
\end{lemma}

This can be done naturally by taking a ``relative bar resolution'', where we define $X_{-1}(M)\\:=I_{-1}=M$, and then $X_{n}:=I_{n}=\ind{H}{G}\cur{I_{n-1}}\cong I_{n-1}\otimes k[G]^H$ for $n\geq0$. This process is already well understood and included in \cite[pp. 273-274]{Kim65}, so we will omit the details. For future reference, we will call this ``relative bar resolution'' a $(G,H)$-bar resolution of $M$.
\subsection{}

An important tool that will be needed later on is the ability to lift a map of modules to a chain map. Namely, we have the following:

\begin{lemma}\label{chmfir}
    Let
    \[
    \adjustbox{scale=1}{%
    \begin{tikzcd}
        0\arrow[r] & N\arrow[r, "d_{-1}"] & A_{0}\arrow[r, "d_{0}"] & A_{1}\arrow[r, "d_{1}"] & \cdots 
    \end{tikzcd}
    }
    \]
    be a $(G,H)$-exact sequence and $f:N\to M$ be any $G$-homomorphism, then for any $(G,H)$-injective resolution of $M$, there exists a commutative diagram
    \[
    \adjustbox{scale=1}{%
    \begin{tikzcd}
        0\arrow[r] & N\arrow[r, "d_{-1}"]\arrow[d, "f"] & A_{0}\arrow[r, "d_{0}"]\arrow[d, "\alpha_0"] & A_{1}\arrow[r, "d_{1}"]\arrow[d, "\alpha_1"] & \cdots \\
        0\arrow[r] & M\arrow[r, "\tilde d_{-1}"] & X_{0}\arrow[r, "\tilde d_{0}"] & X_{1}\arrow[r, "\tilde d_{1}"] & \cdots 
    \end{tikzcd}
    }.
    \]
    Furthermore, given
    \[
    \adjustbox{scale=1}{%
    \begin{tikzcd}
        0\arrow[r] & N\arrow[r, "d_{-1}"]\arrow[d, "f"] & A_{0}\arrow[r, "d_{0}"] & A_{1}\arrow[r, "d_{1}"] & \cdots \\
        0\arrow[r] & M\arrow[r, "\tilde d_{-1}"] & X_{0}\arrow[r, "\tilde d_{0}"] & X_{1}\arrow[r, "\tilde d_{1}"] & \cdots 
    \end{tikzcd}
    },
    \]
    where the top row is a $(G,H)$-exact sequence, the bottom row is a $(G,H)$-injective resolution of $M$, two different lifts of $f$ will result in chain maps which are homotopic. 
\end{lemma}

Proving the above is exactly the same as the in the normal case, so we will omit the details.
\subsection{}

In order to construct a relative Grothendieck Spectral Sequence, it will be important that given a relative long exact sequence of modules, one has a double complex, given by taking relative injective resolutions of each module in the long exact sequence, such that each row will be relative exact.

\begin{proposition}\label{lcir}
    Let
    \[
    \adjustbox{scale=1}{%
    \begin{tikzcd}
        0\arrow[r] & M_{0}\arrow[r, "d_{0}"] & M_{1}\arrow[r, "d_{1}"] & M_{2}\arrow[r, "d_{2}"] & \cdots 
    \end{tikzcd}
    }
    \]
    be a $(G,H)$-exact sequence. Then there exists $(G,H)$-injective resolutions of each $M_{n}$ making the following diagram commute,
    \[
    \adjustbox{scale=1}{%
    \begin{tikzcd}
         & 0\arrow[d] & 0\arrow[d] & 0\arrow[d] & \\
        0\arrow[r] & M_{0}\arrow[r, "d_{0}"]\arrow[d] & M_{1}\arrow[r, "d_{1}"]\arrow[d] & M_{2}\arrow[r, "d_{2}"]\arrow[d] & \cdots \\
        0\arrow[r] & I_{0,0}\arrow[r]\arrow[d] & I_{1,0}\arrow[r]\arrow[d] & I_{2,0}\arrow[r]\arrow[d] & \cdots \\
        0\arrow[r] & I_{0,1}\arrow[r]\arrow[d] & I_{1,1}\arrow[r]\arrow[d] & I_{2,1}\arrow[r]\arrow[d] & \cdots \\
         & \vdots & \vdots & \vdots & 
    \end{tikzcd}
    }
    \]
    and such that each row is $(G,H)$-exact.
\end{proposition}

\begin{proof}
    Let $I_{n,j}=X_{j}\cur{M_{n}}$, as in Section~\ref{rird}, for all $n$ and $j$. Then consider the maps $d_{n,j}:X_{j}\cur{M_{n}}\to X_{j}\cur{M_{n+1}}$ where $d_{n,j}\cur{m\otimes f_1\otimes\cdots\otimes f_{j+1}}=d_{n}(m)\otimes f_1\otimes\cdots\otimes f_{j+1}$. Hence, by the assumption of $d_{n}$,
    \[
    \adjustbox{scale=1}{%
    \begin{tikzcd}
        0\arrow[r] & X_{j}\cur{M_{0}}\arrow[r, "d_{0,j}"] & X_{j}\cur{M_{1}}\arrow[r, "d_{1,j}"] & X_{j}\cur{M_{2}}\arrow[r, "d_{2,j}"] & \cdots 
    \end{tikzcd}
    }
    \]
    is $(G,H)$-exact.

    Next, one can see that
    \begin{align*}
        \del_{n+1,j}\cur{d_{n,j}\cur{m\otimes f_1\otimes\cdots\otimes f_{j+1}}}=&\del_{n+1,j}\cur{d_{n}\cur{m}\otimes f_1\otimes\cdots\otimes f_{j+1}}\\
        =&d_{n}\cur{m}\otimes1\otimes f_{1}\otimes\cdots\otimes f_{j+1}\\
        &+\suml{i=1}{j}(-1)^{i}d_{n}\cur{m}\otimes f_{1}\otimes\cdots\otimes f_{i}\otimes1\otimes f_{i+1}\cdots\otimes f_{j+1}\\
        &+(-1)^{j+1}d_{n}\cur{m}\otimes f_{1}\otimes\cdots\otimes\cdots\otimes f_{j+1}\otimes1\\
        =&d_{n,j+1}\cur{\del_{n,j}\cur{m\otimes f_1\otimes\cdots\otimes f_{j+1}}},
    \end{align*}
    where $\del_{n,j}$ is the map which comes from defining a $(G,H)$-bar resolution, so the diagram commutes.
\end{proof}

When constructing a relative Grothendieck Spectral Sequence, it will be important to have such a construction in the case where we start with a $(G,H)$ short exact sequence, and obtain each row being split-exact. Namely, we will need to following:

\begin{lemma}\label{cir}\textup{\cite[p. 90]{Kim66}}
    Let
    \[
    \adjustbox{scale=1}{%
    \begin{tikzcd}
        0\arrow[r] & M_{0}\arrow[r, "d_{0}"] & M_{1}\arrow[r, "d_{1}"] & M_{2}\arrow[r] & 0 
    \end{tikzcd}
    }
    \]
    be a $(G,H)$-exact sequence. Then there exists $(G,H)$-injective resolutions of each $M_{n}$ making the following diagram commute,
    \[
    \adjustbox{scale=1}{%
    \begin{tikzcd}
         & 0\arrow[d] & 0\arrow[d] & 0\arrow[d] & \\
        0\arrow[r] & M_{0}\arrow[r, "d_{0}"]\arrow[d] & M_{1}\arrow[r, "d_{1}"]\arrow[d] & M_{2}\arrow[r]\arrow[d] & 0 \\
        0\arrow[r] & I_{0,0}\arrow[r]\arrow[d] & I_{1,0}\arrow[r]\arrow[d] & I_{2,0}\arrow[r]\arrow[d] & 0 \\
        0\arrow[r] & I_{0,1}\arrow[r]\arrow[d] & I_{1,1}\arrow[r]\arrow[d] & I_{2,1}\arrow[r]\arrow[d] & 0 \\
         & \vdots & \vdots & \vdots & 
    \end{tikzcd}
    }
    \]
    and such that each row is $(G,H)$-exact.

    Furthermore,
    \[
    \adjustbox{scale=1}{%
    \begin{tikzcd}
        0\arrow[r] & I_{0,j}\arrow[r] & I_{1,j}\arrow[r] & I_{2,j}\arrow[r] & 0
    \end{tikzcd}
    }
    \]
    splits for all $j$.
\end{lemma}

\begin{proof}
    Because of \Cref{lcir}, we only need to show that
    \begin{equation}\label{eq:1}
        \adjustbox{scale=1}{%
        \begin{tikzcd}
        0\arrow[r] & X_{j}\cur{M_{0}}\arrow[r, "d_{0,j}"] & X_{j}\cur{M_{1}}\arrow[r,   "d_{1,j}"] & X_{j}\cur{M_{2}}\arrow[r] & 0 
        \end{tikzcd}
        }
    \end{equation}
    splits. But since Sequence~(\ref{eq:1}) is $(G,H)$-exact and $X_{j}\cur{M_{0}}$ is $(G,H)$-injective, the result follows from \Cref{srses}.
\end{proof}
\subsection{Induction of Relative Exact Sequences}

We now want to prove what will be a important result in order to have a Grothendieck spectral sequence. But first we state a already known result.

\begin{theorem}\label{riiir}\textup{\cite[Theorem 4.1]{CPS83}}
    Let $k$ be algebraically closed, $H$ and $K$ be closed subgroups of $G$ such that $HK=G$, and $M$ be an $H$-module, then $$\res{K}{G}\cur{\ind{H}{G}\cur{M}}\cong\ind{H\cap K}{K}\cur{\res{H\cap K}{H}\cur{M}}.$$
\end{theorem}

Now, using this result, we can show that under the right conditions, induction will take relative exact sequences to relative exact sequences.

\begin{theorem}\label{iores}
    Let $k$ be algebraically closed, $H$ and $K$ be closed subgroups of $G$ such that $HK=G$, and
    \[
    \adjustbox{scale=1}{%
    \begin{tikzcd}
        \cdots\arrow[r] & M_{i-1}\arrow[r, "d_{i-1}"] & M_{i}\arrow[r, "d_{i}"] & M_{i+1}\arrow[r, "d_{i+1}"] & \cdots 
    \end{tikzcd}
    }
    \]
    be an $(H,H\cap K)$-exact sequence, then
    \[
    \adjustbox{scale=1}{%
    \begin{tikzcd}
        \cdots\arrow[r] & \ind{H}{G}\cur{M_{i-1}}\arrow[r] & \ind{H}{G}\cur{M_{i}}\arrow[r] & \ind{H}{G}\cur{M_{i+1}}\arrow[r] & \cdots 
    \end{tikzcd}
    }
    \]
    is $(G,K)$-exact.
\end{theorem}

\begin{proof}
    Consider an $(H,H\cap K)-$exact sequence
    \[
    \adjustbox{scale=1}{%
    \begin{tikzcd}
        \cdots\arrow[r] & M_{i-1}\arrow[r, "d_{i-1}"] & M_{i}\arrow[r, "d_{i}"] & M_{i+1}\arrow[r, "d_{i+1}"] & \cdots. 
    \end{tikzcd}
    }
    \]
    So, $\Ker\cur{d_{i}}$ is an $H\cap K$ direct summand of $M_{i}$, and since induction commutes with direct sums, $\ind{H\cap K}{K}\cur{\res{H\cap K}{H}\cur{\Ker\cur{d_{i}}}}$ is a $K$ direct summand of $\ind{H\cap K}{K}\cur{\res{H\cap K}{H}\cur{M_{i}}}$. Therefore, the sequence
    \begin{equation}\label{eq:2}
    \adjustbox{scale=1}{%
    \begin{tikzcd}
        \cdots\arrow[r] & \ind{H\cap K}{K}\cur{\res{H\cap K}{H}\cur{M_{i-1}}}\arrow[r]\arrow[d, phantom, ""{coordinate, name=Z0}] & \ind{H\cap K}{K}\cur{\res{H\cap K}{H}\cur{M_{i}}}\arrow[dl, rounded corners, to path={ -- ([xshift=2ex]\tikztostart.east)|- (Z0) [near end]\tikztonodes-|([xshift=-2ex]\tikztotarget.west) -- (\tikztotarget)}] \\
        & \ind{H\cap K}{K}\cur{\res{H\cap K}{H}\cur{M_{i+1}}}\arrow[r] & \cdots 
    \end{tikzcd}
    }
    \end{equation}
    is $(K,K)$-exact. Now, by \Cref{riiir}, Sequence~(\ref{eq:2}) is isomorphic to
    \[
    \adjustbox{scale=1}{%
    \begin{tikzcd}
        \cdots\arrow[r] & \res{K}{G}\cur{\ind{H}{G}\cur{M_{i-1}}}\arrow[r]\arrow[d, phantom, ""{coordinate, name=Z0}] & \res{K}{G}\cur{\ind{H}{G}\cur{M_{i}}}\arrow[dl, rounded corners, to path={ -- ([xshift=2ex]\tikztostart.east)|- (Z0) [near end]\tikztonodes-|([xshift=-2ex]\tikztotarget.west) -- (\tikztotarget)}] \\
        & \res{K}{G}\cur{\ind{H}{G}\cur{M_{i+1}}}\arrow[r] & \cdots 
    \end{tikzcd}
    }
    \]
    which will also be $(K,K)$-exact. So clearly
    \[
    \adjustbox{scale=1}{%
    \begin{tikzcd}
        \cdots\arrow[r] & \ind{H}{G}\cur{M_{i-1}}\arrow[r] & \ind{H}{G}\cur{M_{i}}\arrow[r] & \ind{H}{G}\cur{M_{i+1}}\arrow[r] & \cdots 
    \end{tikzcd}
    }
    \]
    is $(G,K)$-exact.
\end{proof}

Applying \Cref{iriri} and \Cref{iores} we obtain the following:

\begin{corollary}\label{rirx2}
    Let $k$ be algebraically closed, and $H$ and $K$ be closed subgroups of $G$ such that $HK=G$. For $M$ an $H$-module and $I_{\bigcdot}$ an $(H,H\cap K)$-injective resolution of $M$, $\ind{H}{G}\cur{I_{\bigcdot}}$ is a $(G,K)$-injective resolution of $\ind{H}{G}\cur{M}$.
\end{corollary}
\subsection{}


The following proposition should be compared to Proposition~4.2, Proposition~4.3, and Proposition~4.4 in \cite{Kim66}. Note that we have different conditions on our subgroups than that of \cite{Kim66}, and when comparing to the later two, note that our result is for arbitrary characteristic, while Kimura's are only for characteristic zero.

\begin{proposition}
    Let $H$ and $K$ be closed subgroups of $G$ such that $HK=G$, and let $I$ be $(G,K)$-injective. Then $I$ is $(H,H\cap K)$-injective. Furthermore, for any $H$-module $N$, we have that $I\otimes N$ is $(H,H\cap K)$-injective.
\end{proposition}

\begin{proof}
    First, if $I$ is $(H,H\cap K)$-injective, from \Cref{triri} it is clear that for any $H$-module $N$, $I\otimes N$ is $(H,H\cap K)$-injective. So, we only need to show that $I$ is $(H,H\cap K)$-injective. By \Cref{riiffds}, $I$ is a direct summand of some $\ind{K}{G}(M)$, so clearly $\res{H}{G}(I)$ is a direct summand of $\res{H}{G}\cur{\ind{K}{G}(M)}$. Now by \Cref{riiir}, $\res{H}{G}\cur{\ind{K}{G}(M)}\cong\ind{H\cap K}{H}\cur{\res{H\cap K}{K}(M)}$, which is $(H,H\cap K)$-injective by \Cref{iri}. Therefore, from \Cref{dsim}, $I$ is $(H,H\cap K)$-injective.
\end{proof}
\section{Relative Right Derived Functor}\label{Relative Right Derived Functor}

\subsection{}\label{relative cohomology}

Let $G$ and $G'$ be algebraic groups, $H$ a closed subgroup of $G$, and $F$ be an additive left exact covariant functor from $G$-modules to $G'$-modules. The definition of a relative cohomology was first given by Kimura in \cite[p. 273]{Kim65}. We note that Kimura only gave the definition for when $F=\Hom{G}{M,-}$, but the same definition works more generally. For convenience, we will include the definition here as well. Define a relative right derived functor of $F$ in the following way. Let $M$ be a $G$-module and consider a $(G,H)$-injective resolution
\[
\adjustbox{scale=1}{%
\begin{tikzcd}
    0\arrow[r] & M\arrow[r] & I_{0}\arrow[r] & I_{1}\arrow[r] & \cdots. 
\end{tikzcd}
}
\]
Apply $F$ to get the complex
\[
\adjustbox{scale=1}{%
\begin{tikzcd}
    0\arrow[r] & F\cur{I_{0}}\arrow[r] & F\cur{I_{1}}\arrow[r] & \cdots. 
\end{tikzcd}
}
\]
Define $\textup{R}_{(G,H)}^{i}F\cur{M}$, to be the $i$-th cohomology of the above complex. Note that if $k$ is algebraically closed, then $\textup{R}_{(G,\{1\})}^{i}F\cur{M}=\textup{R}^{i}F\cur{M}$.

For the case when $F=\Hom{G}{N,-}$, denote $\textup{R}_{(G,H)}^{i}\Hom{G}{N,M}$ as $\Ext{(G,H)}{i}\cur{N,M}$.
\subsection{}

Similarly to $\textup{R}^{i}F\cur{M}$, it can be shown that $\textup{R}_{(G,H)}^{i}F\cur{M}$ is well defined, $\textup{R}_{(G,H)}^{0}F\cur{M}\cong F\cur{M}$, and given a short exact sequence of $G$-modules, there is a long exact sequence of $G'$-modules. For completeness we will include the statements and proofs.

\begin{lemma}\label{rrdfawd}
    $\textup{R}_{(G,H)}^{i}F\cur{M}$ is well defined (i.e. independent of choice of injective resolution) for $i\geq0$.
\end{lemma}

\begin{proof}
    Consider two $(G,H)$-injective resolutions of $M$,
    \[
    \adjustbox{scale=1}{%
    \begin{tikzcd}
        0\arrow[r] & M\arrow[r] & I_{0}\arrow[r] & I_{1}\arrow[r] & \cdots 
    \end{tikzcd}
    }
    \]
    and
    \[
    \adjustbox{scale=1}{%
    \begin{tikzcd}
        0\arrow[r] & M\arrow[r] & J_{0}\arrow[r] & J_{1}\arrow[r] & \cdots. 
    \end{tikzcd}
    }
    \]
    Next, consider the identity map, $\id_{M}:M\to M$, which lifts by \Cref{chmfir} to give chain maps $f$ and $g$ making the following diagram commute
    \[
    \adjustbox{scale=1}{%
    \begin{tikzcd}
        0\arrow[r] & I_{0}\arrow[r]\arrow[d, "f_{0}"] & I_{1}\arrow[r]\arrow[d, "f_{1}"] & I_{2}\arrow[r]\arrow[d, "f_{2}"] & \cdots \\
        0\arrow[r] & J_{0}\arrow[r]\arrow[d, "g_{0}"] & J_{1}\arrow[r]\arrow[d, "g_{1}"] & J_{2}\arrow[r]\arrow[d, "g_{2}"] & \cdots \\
        0\arrow[r] & I_{0}\arrow[r] & I_{1}\arrow[r] & I_{2}\arrow[r] & \cdots 
    \end{tikzcd}
    }.
    \]
    Now, $g\circ f$ is also a chain map that lifts from $\id_{M}$, but the identity map, i.e., $\id_{n}:I_n\to I_n$ for all $n$, is also a lift of $\id_{M}$. Hence, $g\circ f$ is homotopic to $\id_{I_{\bigcdot}}$. Now apply $F$ to everything, so $F\cur{g\circ f}$ and $F\cur{\id_{I_{\bigcdot}}}$ will also be homotopic and induce the same maps on homology. Therefore, the following diagram commutes
    \[
    \adjustbox{scale=1}{%
    \begin{tikzcd}
        H^{n}\cur{F\cur{I_{\bigcdot}}}\arrow[r, "F\cur{f}"]\arrow[dr, "F\cur{\id_{I_{\bigcdot}}}", swap] & H^{n}\cur{F\cur{J_{\bigcdot}}}\arrow[d, "F\cur{g}"] \\
         & H^{n}\cur{F\cur{I_{\bigcdot}}}
    \end{tikzcd}
    }.
    \]
    Using the same argument shows that the diagram
    \[
    \adjustbox{scale=1}{%
    \begin{tikzcd}
        H^{n}\cur{F\cur{J_{\bigcdot}}}\arrow[r, "F\cur{g}"]\arrow[dr, "F\cur{\id_{J_{\bigcdot}}}", swap] & H^{n}\cur{F\cur{I_{\bigcdot}}}\arrow[d, "F\cur{f}"] \\
         & H^{n}\cur{F\cur{J_{\bigcdot}}}
    \end{tikzcd}
    }
    \]
    also commutes and so $H^{n}\cur{F\cur{I_{\bigcdot}}}\cong H^{n}\cur{F\cur{J_{\bigcdot}}}$. Therefore, $\textup{R}_{(G,H)}^{i}F\cur{M}$ is well defined.
\end{proof}

\begin{observation}\label{rf0}
    $\textup{R}_{(G,H)}^{0}F\cur{M}\cong F\cur{M}$.
\end{observation}

\begin{proof}
    Consider a $(G,H)$-injective resolution of $M$,
    \[
    \adjustbox{scale=1}{%
    \begin{tikzcd}
        0\arrow[r] & M\arrow[r, "d_{-1}"] & I_{0}\arrow[r, "d_{0}"] & I_{1}\arrow[r, "d_{1}"] & \cdots, 
    \end{tikzcd}
    }
    \]
    and apply $F$ to get 
    \[
    \adjustbox{scale=1}{%
    \begin{tikzcd}
        0\arrow[r] & F\cur{M}\arrow[r, "F\cur{d_{-1}}"] & F\cur{I_{0}}\arrow[r, "F\cur{d_{0}}"] & F\cur{I_{1}}\arrow[r, "F\cur{d_{1}}"] & \cdots. 
    \end{tikzcd}
    }
    \]
    $\textup{R}_{(G,H)}^{0}F\cur{M}=\Ker\cur{F\cur{d_{0}}}$, and since $F$ is left exact, $\Ker\cur{F\cur{d_{0}}}=\IM\cur{F\cur{d_{-1}}}=F\cur{M}$.
\end{proof}

\begin{proposition}\label{rles1}
    Let $F$ be an additive left exact covariant functor from $G$-modules to $G'$-modules and
    \[
    \adjustbox{scale=1}{%
    \begin{tikzcd}
        0\arrow[r] & M'\arrow[r, "g"] & M\arrow[r, "h"] & M''\arrow[r] & 0 
    \end{tikzcd}
    }
    \]
    be a $(G,H)$-exact sequence. Then there exists a long exact sequence of $G'$-modules
    \[
    \adjustbox{scale=1}{%
    \begin{tikzcd}
        0\arrow[r] & F\cur{M'}\arrow[r] & F\cur{M}\arrow[r]\arrow[d, phantom, ""{coordinate, name=Z}] & F\cur{M''}\arrow[dll, rounded corners, to path={ -- ([xshift=2ex]\tikztostart.east)|- (Z) [near end]\tikztonodes-|([xshift=-2ex]\tikztotarget.west) -- (\tikztotarget)}] & \\
         & \textup{R}_{(G,H)}^{1}F\cur{M'}\arrow[r] & \textup{R}_{(G,H)}^{1}F\cur{M}\arrow[r] & \textup{R}_{(G,H)}^{1}F\cur{M''}\arrow[r] & \cdots 
    \end{tikzcd}
    }.
    \]
\end{proposition}

\begin{proof}
    By \Cref{cir}, there exists a $(G,H)$-injective resolutions of $M'$, $M$, and $M''$ making the following diagram commute,
    \[
    \adjustbox{scale=1}{%
    \begin{tikzcd}
         & 0\arrow[d] & 0\arrow[d] & 0\arrow[d] & \\
        0\arrow[r] & M'\arrow[r, "g"]\arrow[d] & M\arrow[r, "h"]\arrow[d] & M''\arrow[r]\arrow[d] & 0 \\
        0\arrow[r] & I_{0}'\arrow[r]\arrow[d] & I_{0}\arrow[r]\arrow[d] & I_{0}''\arrow[r]\arrow[d] & 0 \\
        0\arrow[r] & I_{1}'\arrow[r]\arrow[d] & I_{1}\arrow[r]\arrow[d] & I_{1}''\arrow[r]\arrow[d] & 0 \\
         & \vdots & \vdots & \vdots & 
    \end{tikzcd}
    }
    \]
    such that each row is $(G,H)$-exact and such that
    \[
    \adjustbox{scale=1}{%
    \begin{tikzcd}
        0\arrow[r] & I_{n}'\arrow[r] & I_{n}\arrow[r] & I_{n}''\arrow[r] & 0
    \end{tikzcd}
    }
    \]
    splits for all $n$.

    Now, apply $F$ to get the commutative diagram
    \[
    \adjustbox{scale=1}{%
    \begin{tikzcd}
         & 0\arrow[d] & 0\arrow[d] & 0\arrow[d] & \\
        0\arrow[r] & F\cur{I_{0}'}\arrow[r]\arrow[d] & F\cur{I_{0}}\arrow[r]\arrow[d] & F\cur{I_{0}''}\arrow[r]\arrow[d] & 0 \\
        0\arrow[r] & F\cur{I_{1}'}\arrow[r]\arrow[d] & F\cur{I_{1}}\arrow[r]\arrow[d] & F\cur{I_{1}''}\arrow[r]\arrow[d] & 0 \\
        0\arrow[r] & F\cur{I_{2}'}\arrow[r]\arrow[d] & F\cur{I_{2}}\arrow[r]\arrow[d] & F\cur{I_{2}''}\arrow[r]\arrow[d] & 0 \\
         & \vdots & \vdots & \vdots & 
    \end{tikzcd}
    },
    \]
    where each row is exact. Hence, via applying the Snake Lemma, there exists a long exact sequence
    \[
    \adjustbox{scale=1}{%
    \begin{tikzcd}
        0\arrow[r] & F\cur{M'}\arrow[r] & F\cur{M}\arrow[r]\arrow[d, phantom, ""{coordinate, name=Z}] & F\cur{M''}\arrow[dll, rounded corners, to path={ -- ([xshift=2ex]\tikztostart.east)|- (Z) [near end]\tikztonodes-|([xshift=-2ex]\tikztotarget.west) -- (\tikztotarget)}] & \\
         & \textup{R}_{(G,H)}^{1}F\cur{M'}\arrow[r] & \textup{R}_{(G,H)}^{1}F\cur{M}\arrow[r] & \textup{R}_{(G,H)}^{1}F\cur{M''}\arrow[r] & \cdots 
    \end{tikzcd}
    }.
    \]
\end{proof}
\subsection{}

Under the right conditions, taking the relative right derived functor will be the same as the normal right derived functor. This provides an interesting contrast when compared to the relative cohomology of Lie algebras. There is no criterion for which the relative cohomology of Lie algebras will be the same as the normal cohomology of Lie algebras. But, a consequence of the following result is that there is criterion for which the relative cohomology of algebraic groups will be the same as the normal cohomology of algebraic groups.

\begin{proposition}\label{rtt}
    Let $F$ be an additive left exact covariant functor from $G$-modules to $G'$-modules and $H$ be a closed subgroup of $G$ such that $\Mod{H}$ is semisimple. Then $\textup{R}_{(G,H)}^{i}F\cur{M}=\textup{R}^{i}F\cur{M}$ for $i\geq0$.
\end{proposition}

\begin{proof}
    Consider an injective resolution of $M$,
    \[
    \adjustbox{scale=1}{%
    \begin{tikzcd}
        0\arrow[r] & M\arrow[r, "d_{-1}"] & I_{0}\arrow[r, "d_{0}"] & I_{1}\arrow[r, "d_{1}"] & \cdots. 
    \end{tikzcd}
    }
    \]
    By \Cref{rimu}, each $I_{i}$ is $(G,H)$-injective, and since $\Mod{H}$ is semisimple, the sequence splits as an $H$-module, and so is $(G,H)$-exact. Hence the injective resolution is also a $(G,H)$-injective resolution, and so the result follows.
\end{proof}

Note that this means under these conditions any result for $\textup{R}^{i}F\cur{M}$ will be true for $\textup{R}_{(G,H)}^{i}F\cur{M}$.
\subsection{Relative Acyclic Modules}

We call a $G$-module $M$, $(G,H)$-acyclic to a functor $F:\Mod{G}\to\Mod{G'}$ if $\textup{R}_{(G,H)}^{i}F\cur{M}=0$ for all $i>0$. The following two results are basic generalizations that will be useful for the general theory.

\begin{lemma}\label{rif0}
    If $M$ is $(G,H)$-injective, then $\textup{R}_{(G,H)}^{i}F\cur{M}=0$ for all $i>0$.
\end{lemma}

\begin{proof}
    Consider the exact sequence
    \[
    \adjustbox{scale=1}{%
    \begin{tikzcd}
        0\arrow[r] & M\arrow[r, "\id_{M}"] & M\arrow[r] & 0\arrow[r] & \cdots.
    \end{tikzcd}
    }
    \]

    It is clear that this is $(G,H)$-exact, and thus a $(G,H)$-injective resolution of $M$. Now apply $F$ to the sequence to get
    \[
    \adjustbox{scale=1}{%
    \begin{tikzcd}
        0\arrow[r] & F\cur{M}\arrow[r] & 0\arrow[r] & \cdots.
    \end{tikzcd}
    }
    \]

    Hence, it is clear that for $i>0$, $\textup{R}_{(G,H)}^{i}F\cur{M}=0$.
\end{proof}

\begin{lemma}
    Let $M$ be a direct summand of $V$, so $V\cong M\oplus N$. If $\textup{R}_{(G,H)}^{i}F(V)=0$ for $i>0$, then $\textup{R}_{(G,H)}^{i}F(M)=0$ for $i>0$.
\end{lemma}

\begin{proof}
    Let $\cur{X_{\bigcdot}(M),d_{M,\bigcdot}}$ and $\cur{X_{\bigcdot}(N),d_{N,\bigcdot}}$ be $(G,H)$-injective resolutions of $M$ and $N$ respectively. By \Cref{dsri} $\cur{X_{\bigcdot}(M)\oplus X_{\bigcdot}(N),d_{M,\bigcdot}\oplus d_{N,\bigcdot}}$ is a $(G,H)$-injective resolution of $V$. To compute $\textup{R}_{(G,H)}^{i}F(V)$, we compute the cohomology of the complex\\ $\cur{F\cur{X_{\bigcdot}(M)\oplus X_{\bigcdot}(N)},F\cur{d_{M,\bigcdot}\oplus d_{N,\bigcdot}}}\cong\cur{F\cur{X_{\bigcdot}(M)}\oplus F\cur{X_{\bigcdot}(N)},F\cur{d_{M,\bigcdot}}\oplus F\cur{d_{N,\bigcdot}}}$. So it is clear that if $\textup{R}_{(G,H)}^{i}F(V)=0$ for $i>0$, then $\textup{R}_{(G,H)}^{i}F(M)=0$ for $i>0$.
\end{proof}
\subsection{Relative Ext}

We now focus on when $F=\Hom{G}{M,-}$. Recall the notation for $\textup{R}_{(G,H)}^{i}\Hom{G}{M,N}$ is $\Ext{(G,H)}{i}\cur{M,N}$.

\begin{proposition}\textup{\cite[Proposition 2.1]{Kim66}}\label{rles2}
    Let $N$ be a $G$-module, $H$ be a closed subgroup of $G$, and 
    \[
    \adjustbox{scale=1}{%
    \begin{tikzcd}
        0\arrow[r] & M'\arrow[r, "g"] & M\arrow[r, "h"] & M''\arrow[r] & 0 
    \end{tikzcd}
    }
    \]
    a $(G,H)$-exact sequence. Then there exists a long exact sequence
    \[
    \adjustbox{scale=1}{%
    \begin{tikzcd}
        0\arrow[r] & \Hom{G}{M'',N}\arrow[r] & \Hom{G}{M,N}\arrow[r]\arrow[d, phantom, ""{coordinate, name=Z}] & \Hom{G}{M',N}\arrow[dll, rounded corners, to path={ -- ([xshift=2ex]\tikztostart.east)|- (Z) [near end]\tikztonodes-|([xshift=-2ex]\tikztotarget.west) -- (\tikztotarget)}] & \\
         & \Ext{(G,H)}{1}\cur{M'',N}\arrow[r] & \Ext{(G,H)}{1}\cur{M,N}\arrow[r] & \Ext{(G,H)}{1}\cur{M',N}\arrow[r] & \cdots 
    \end{tikzcd}
    }.
    \]
\end{proposition}

The proof of \Cref{rles2} is similar to \Cref{rles1}, where one takes a $(G,H)$-injective resolution, $I^{\bigcdot}$, of $N$, then apply $\Hom{G}{-,I^{\bigcdot}}$ to get a double complex and last apply the Snake Lemma. Since this is similar, we will  omit the details.

It is well known that a $G$-module $N$ is injective if and only if $\Hom{G}{-,N}$ is exact. We now want to give the relative analog of this result. So a $G$-module $N$ is $(G,H)$-injective if and only if $\Hom{G}{-,N}$ exact on $(G,H)$-exact sequences. This can be formulated as follows:

\begin{proposition}\label{rie0}
    Let $N$ be a $G$-module. Then the following are equivalent:
    \begin{enumerate}[(a)]
        \item $N$ is $(G,H)$-injective.
        \vspace{0.1cm}
        \item For any $G$-module $M$, $\Ext{(G,H)}{i}\cur{M,N}=0$ for all $i>0$.
        \item For any $G$-module $M$, $\Ext{(G,H)}{1}\cur{M,N}=0$.
    \end{enumerate}
\end{proposition}

\begin{proof}
    For $(a)\Rightarrow(b)$ apply \Cref{rif0} to $F=\Hom{G}{M,-}$. $(b)\Rightarrow(c)$ is obvious.

    For $(c)\Rightarrow(a)$ consider a $(G,H)$-exact sequence
    \[
    \adjustbox{scale=1}{%
    \begin{tikzcd}
        0\arrow[r] & M_{1}\arrow[r, "d_{1}"] & M_{2}\arrow[r, "d_{2}"] & M_{3}\arrow[r] & 0.
    \end{tikzcd}
    }
    \]
    Now, apply \Cref{rles2} to get the exact sequence
    \[
    \adjustbox{scale=1}{%
    \begin{tikzcd}
        0\arrow[r] & \Hom{G}{M_{3},N}\arrow[r, "\tilde d_{2}"] & \Hom{G}{M_{2},N}\arrow[r, "\tilde d_{1}"]\arrow[d, phantom, ""{coordinate, name=Z0}] & \Hom{G}{M_{1},N}\arrow[dll, rounded corners, to path={ -- ([xshift=2ex]\tikztostart.east)|- (Z0) [near end]\tikztonodes-|([xshift=-2ex]\tikztotarget.west) -- (\tikztotarget)}] \\
        & \Ext{(G,H)}{1}\cur{M_{3},N} & \phantom{blah}
    \end{tikzcd}
    }.
    \]
    From our assumption, $\Ext{(G,H)}{1}\cur{M_{3},N}=0$, and so the map $\tilde d_{1}=-\circ d_{1}$ is surjective. Hence, for any $f:M_{1}\to N$, there exists an $h:M_{2}\to N$ such that $f=h\circ d_{1}$. Therefore, $N$ is $(G,H)$-injective.
\end{proof}
\subsection{}

Continuing with the theme of finding relative Ext vanishing criteria for relative injective modules, we now state when the result that $M$ is injective if and only if $\Ext{G}{1}\cur{V,M}=0$ for all finite dimensional modules $V$ generalizes to the relative case. To do so, we will need to use the Mittag-Leffler condition on a tower of cochain complexes of abelian groups. The reader is referred to \cite[Section 3.5]{Wei94} for the definitions and a detailed overview.

\begin{proposition}\label{kachssriwefd0}
    Let $H$ be a closed subgroup of $G$ such that $\Mod{H}$ is semisimple. In the category $\Mod{G}$, if $\Ext{(G,H)}{1}\cur{V_{i},N}=0$ for all finite dimensional $G$-modules $V_{i}$, then $N$ is $(G,H)$-injective.
\end{proposition}

\begin{proof}
    We will use that given a tower of cochain complexes of abelian groups,
    \begin{equation*}
        \adjustbox{scale=1}{
        \begin{tikzcd}
            \cdots\arrow[r] & C_{1}\arrow[r] & C_{0},
        \end{tikzcd}
        }
    \end{equation*}
    which satisfies the Mittag-Leffler condition, there exists, for each $q$, an exact sequence
    \begin{equation*}
        \adjustbox{scale=1}{
        \begin{tikzcd}
            0\arrow[r] & {\lim\limits_{\longleftarrow}}^{1}H^{q-1}\cur{C_{i}}\arrow[r] & H^{q}\cur{\lim\limits_{\longleftarrow}C_{i}}\arrow[r] & \lim\limits_{\longleftarrow}H^{q}\cur{C_{i}}\arrow[r] & 0.
        \end{tikzcd}
        }
    \end{equation*}
    For a proof see \cite[Theorem 3.5.8]{Wei94}.

    Let
    \begin{equation*}
        \adjustbox{scale=1}{
        \begin{tikzcd}
            0\arrow[r] & N\arrow[r] & I_{0}\arrow[r] & I_{1}\arrow[r] & \cdots
        \end{tikzcd}
        }
    \end{equation*}
    be a $(G,H)$-injective resolution of $N$, and $V$ be an infinite dimensional $G$-module. There exists submodules $$0=V_{0}\subseteq V_{1}\subseteq\cdots\subseteq V$$ such that $V_{i}$ is finite dimensional for all $i$ and $\lim\limits_{\longrightarrow}V_{i}=V$. Define
    \begin{equation*}
        \adjustbox{scale=1}{
        \begin{tikzcd}
            C_{i}:=0\arrow[r] & \Hom{G}{V_{i},I_{0}}\arrow[r] & \Hom{G}{V_{i},I_{1}}\arrow[r] & \cdots.
        \end{tikzcd}
        }
    \end{equation*}
    Now we need to show that for $n\geq-1$,
    \begin{equation*}
        \adjustbox{scale=1}{
        \begin{tikzcd}
            \cdots\arrow[r] & \Hom{G}{V_{2},I_{n}}\arrow[r] & \Hom{G}{V_{1},I_{n}}\arrow[r] & 0,
        \end{tikzcd}
        }
    \end{equation*}
    where we define $I_{-1}:=N$, satisfies the Mittag-Leffler condition. Consider the exact sequence
    \[
    \adjustbox{scale=1}{%
    \begin{tikzcd}
        0\arrow[r] & V_{i}\arrow[r, "i"] & V_{i+1}\arrow[r] & V_{i+1}/V_{i}\arrow[r] & 0. 
    \end{tikzcd}
    }
    \]
    Since $\Mod{H}$ is semisimple, the sequence is $(G,H)$-exact, and so by \Cref{rles2} the following sequence is exact
    \[
    \adjustbox{scale=1}{%
    \begin{tikzcd}
        0\arrow[r] & \Hom{G}{V_{i+1}/V_{i},I_{n}}\arrow[r] & \Hom{G}{V_{i+1},I_{n}}\arrow[r, "-\circ i"]\arrow[d, phantom, ""{coordinate, name=Z0}] & \Hom{G}{V_{i},I_{n}}\arrow[dll, rounded corners, to path={ -- ([xshift=2ex]\tikztostart.east)|- (Z0) [near end]\tikztonodes-|([xshift=-2ex]\tikztotarget.west) -- (\tikztotarget)}] \\
        & \Ext{(G,H)}{1}\cur{V_{i+1}/V_{i},I_{n}} & \phantom{blah}
    \end{tikzcd}
    }.
    \]
    From our assumption, $\Ext{(G,H)}{1}\cur{V_{i+1}/V_{i},N}=0$, so when $n=-1$ the map $-\circ i$ is surjective. When $n\geq0$, by \Cref{rie0} one has $\Ext{(G,H)}{1}\cur{V_{i+1}/V_{i},I_{n}}=0$, so $-\circ i$ is again surjective. Hence, for $n\geq-1$,
    \begin{equation*}
        \adjustbox{scale=1}{
        \begin{tikzcd}
            \cdots\arrow[r] & \Hom{G}{V_{2},I_{n}}\arrow[r] & \Hom{G}{V_{1},I_{n}}\arrow[r] & 0
        \end{tikzcd}
        }
    \end{equation*}
    satisfies the Mittag-Leffler condition.

    Now, using the fact that $$\lim\limits_{\longleftarrow}\Hom{G}{V_{i},I_{j}}\cong\Hom{G}{\lim\limits_{\longrightarrow}V_{i},I_{j}},$$ it is clear that
    \begin{equation*}
        \adjustbox{scale=1}{
        \begin{tikzcd}
            \lim\limits_{\longleftarrow}C_{i}=0\arrow[r] & \Hom{G}{\lim\limits_{\longrightarrow}V_{i},I_{0}}\arrow[r] & \cdots.
        \end{tikzcd}
        }
    \end{equation*}
    Therefore, from the definition of $\Ext{(G,H)}{}$ and applying \cite[Theorem 3.5.8]{Wei94}, there exists an exact sequence
    \begin{equation*}
        \adjustbox{scale=1}{
        \begin{tikzcd}
            0\arrow[r] & {\lim\limits_{\longleftarrow}}^{1}\Ext{(G,H)}{q-1}\cur{V_{i},N}\arrow[r]\arrow[d, phantom, ""{coordinate, name=Z0}] & \Ext{(G,H)}{q}\cur{\lim\limits_{\longrightarrow}V_{i},N}\arrow[dl, rounded corners, to path={ -- ([xshift=2ex]\tikztostart.east)|- (Z0) [near end]\tikztonodes-|([xshift=-2ex]\tikztotarget.west) -- (\tikztotarget)}] \\
            & \lim\limits_{\longleftarrow}\Ext{(G,H)}{q}\cur{V_{i},N}\arrow[r] & 0
        \end{tikzcd}
        },
    \end{equation*}
    and in particular when $q=1$, $${\lim\limits_{\longleftarrow}}^{1}\Hom{G}{V_{i},N}\cong\Ext{(G,H)}{1}\cur{\lim\limits_{\longrightarrow}V_{i},N},$$
    since by our assumption $\Ext{(G,H)}{1}\cur{V_{i},N}=0$.
    
    Last, since
    \begin{equation*}
        \adjustbox{scale=1}{
        \begin{tikzcd}
            \cdots\arrow[r] & \Hom{G}{V_{2},N}\arrow[r] & \Hom{G}{V_{1},N}\arrow[r] & 0
        \end{tikzcd}
        }
    \end{equation*}
    satisfies the Mittag-Leffler condition, ${\lim\limits_{\longleftarrow}}^{1}\Hom{G}{V_{i},N}=0$, see \cite[Proposition 3.5.7]{Wei94} for a proof. Hence, for any $G$-module $V$, $\Ext{(G,H)}{1}\cur{V,N}=0$, so by \Cref{rie0} $N$ is $(G,H)$-injective.
\end{proof}

\begin{remark}
    \textup{In \cite[Proposition 4.1]{Kim66} it is claimed that our \Cref{kachssriwefd0} holds for any closed subgroup $H$ of $G$ by using the same proof as in \cite[Proposition 2.1]{Hoc63}. This method of proof is not sufficient though. If trying to generalize the proof of \cite[Proposition 2.1]{Hoc63}, one would need that given a $G$-module $A$, a submodule $C$ such that $C$ is a direct summand of $A$ as an $H$ module, and an element $a\in A\backslash C$, the module $C'$ generated by $C$ and $a$ would need to contain $C$ as a direct summand as an $H$-module, which is not guaranteed.}
\end{remark}
\section{Cohomology}\label{Cohomology}

For the remainder of this paper we will assume that $k$ is algebraically closed.

\subsection{A Relative Grothendieck Spectral Sequence}

Let $G$ and $G'$ be algebraic groups, $H$ a closed subgroup of $G$, $H'$ a closed subgroup of $G'$, and $F:\Mod{G}\to\Mod{G'}$ be a left exact covariant functor. We call $F$ $(H,H')$-{\it split} if given a $(G,H)$-exact sequence,
\[
\adjustbox{scale=1}{%
\begin{tikzcd}
    \cdots\arrow[r] & M_{i-1}\arrow[r, "d_{i-1}"] & M_{i}\arrow[r, "d_{i}"] & M_{i+1}\arrow[r] & \cdots, 
\end{tikzcd}
}
\]
the sequences
\begin{equation}\label{eq:7}
    \adjustbox{scale=1}{
    \begin{tikzcd}
        0\arrow[r] & \Ker\cur{F\cur{d_{j}}}\arrow[r] & F\cur{M_{j}}\arrow[r] & \IM\cur{F\cur{d_{j}}}\arrow[r] & 0
    \end{tikzcd}
    }
\end{equation}
and
\begin{equation}\label{eq:8}
    \adjustbox{scale=1}{
    \begin{tikzcd}
        0\arrow[r] & \IM\cur{F\cur{d_{j}}}\arrow[r] & \Ker\cur{F\cur{d_{j+1}}}\arrow[r] & H^{j+1}\cur{F\cur{M_{\bigcdot}}}\arrow[r] & 0
    \end{tikzcd}
    }
\end{equation}
are $(G',H')$-exact.

\begin{theorem}\label{ggspfrrdf}
    Let $G$, $G'$, and $G''$ be algebraic groups, $H$ a closed subgroup of $G$, $H'$ a closed subgroup of $G'$, and $F:\Mod{G}\to\Mod{G'}$ and $F':\Mod{G'}\to\Mod{G''}$ be additive left exact covariant functors. 
    
    If $F$ maps $(G,H)$-injective objects to $(G',H')$-acyclic objects of $F'$, and is $(H,H')$-split, then for all $G$-modules $M$, there exists a spectral sequence $$E_2^{i,j}=\textup{R}_{(G',H')}^{i}F'\cur{\textup{R}_{(G,H)}^{j}F\cur{M}}\Rightarrow\cur{\textup{R}_{(G,H)}^{i+j}\cur{F'\circ F}}\cur{M}.$$
\end{theorem}

\begin{proof}
    Note that we follow the proof from \cite[Chapter XX, Theorem 9.6]{Lan02}, although in our case we need to be more careful, since it will be necessary for certain sequences to be $(G',H')$-exact. This is why one must have that $F$ is $(H,H')$-split. We remark that if $H'=\crly{1}$, then $F$ will trivially be $(H,\crly{1})$-split, so this is a generalization of \cite[Chapter XX, Theorem 9.6]{Lan02}.

    Let $M$ be a $G$-module, and consider a $(G,H)$-injective resolution
    \[
    \adjustbox{scale=1}{%
    \begin{tikzcd}
        0\arrow[r] & M\arrow[r, "d_{-1}"] & I_{0}\arrow[r, "d_{0}"] & I_{1}\arrow[r, "d_{1}"] & \cdots
    \end{tikzcd}
    }
    \]
    and apply $F$ to get
    \[
    \adjustbox{scale=1}{%
    \begin{tikzcd}
        0\arrow[r] & F\cur{I_{0}}\arrow[r, "F\cur{d_{0}}"] & F\cur{I_{1}}\arrow[r, "F\cur{d_{1}}"] & F\cur{I_{2}}\arrow[r, "F\cur{d_{2}}"] & \cdots.
    \end{tikzcd}
    }
    \]
    From our assumptions, $F\cur{I_{n}}$ is $(G',H')$-acyclic to $F'$ for all $n$.
    
    Now take a $(G',H')$-bar resolution in the same way as in \Cref{lcir} of the complex to get
    \[
    \adjustbox{scale=1}{%
    \begin{tikzcd}
        & 0\arrow[d] &  & 0\arrow[d] &  & 0\arrow[d] &  & \\
        0\arrow[r] & F\cur{I_{0}}\arrow[rr, "F\cur{d_{0}}"]\arrow[d, "\del_{0,-1}"] &  & F\cur{I_{1}}\arrow[rr, "F\cur{d_{1}}"]\arrow[d, "\del_{1,-1}"] &  & F\cur{I_{2}}\arrow[rr, "F\cur{d_{2}}"]\arrow[d, "\del_{2,-1}"] &  & \cdots \\
        0\arrow[r] & I_{0,0}\arrow[rr, "d_{0,0}"]\arrow[d, "\del_{0,0}"] &  & I_{1,0}\arrow[rr, "d_{1,0}"]\arrow[d, "\del_{1,0}"] &  & I_{2,0}\arrow[rr, "d_{2,0}"]\arrow[d, "\del_{2,0}"] &  & \cdots \\
        0\arrow[r] & I_{0,1}\arrow[rr, "d_{0,1}"]\arrow[d, "\del_{0,1}"] &  & I_{1,1}\arrow[rr, "d_{1,1}"]\arrow[d, "\del_{1,1}"] &  & I_{2,1}\arrow[rr, "d_{2,1}"]\arrow[d, "\del_{2,1}"] &  & \cdots \\
         & \vdots &  & \vdots &  & \vdots &  & 
    \end{tikzcd}
    },
    \]
    and define $\mc{I}$ the be the complex
    \[
    \adjustbox{scale=1}{%
    \begin{tikzcd}
        & 0\arrow[d] &  & 0\arrow[d] &  & 0\arrow[d] &  & \\
        0\arrow[r] & I_{0,0}\arrow[rr, "d_{0,0}"]\arrow[d, "\del_{0,0}"] &  & I_{1,0}\arrow[rr, "d_{1,0}"]\arrow[d, "\del_{1,0}"] &  & I_{2,0}\arrow[rr, "d_{2,0}"]\arrow[d, "\del_{2,0}"] &  & \cdots \\
        0\arrow[r] & I_{0,1}\arrow[rr, "d_{0,1}"]\arrow[d, "\del_{0,1}"] &  & I_{1,1}\arrow[rr, "d_{1,1}"]\arrow[d, "\del_{1,1}"] &  & I_{2,1}\arrow[rr, "d_{2,1}"]\arrow[d, "\del_{2,1}"] &  & \cdots \\
         & \vdots &  & \vdots &  & \vdots &  & 
    \end{tikzcd}
    }.
    \]
    Then apply $F'$ to get a double complex, which will act as our $E_{0}$ page,
    \[
    \adjustbox{scale=1}{%
    \begin{tikzcd}
        & 0\arrow[d] &  & 0\arrow[d] &  & 0\arrow[d] &  & \\
        0\arrow[r] & F'\cur{I_{0,0}}\arrow[rr, "F'\cur{d_{0,0}}"]\arrow[d, "F'\cur{\del_{0,0}}"] &  & F'\cur{I_{1,0}}\arrow[rr, "F'\cur{d_{1,0}}"]\arrow[d, "F'\cur{\del_{1,0}}"] &  & F'\cur{I_{2,0}}\arrow[rr, "F'\cur{d_{2,0}}"]\arrow[d, "F'\cur{\del_{2,0}}"] &  & \cdots \\
        0\arrow[r] & F'\cur{I_{0,1}}\arrow[rr, "F'\cur{d_{0,1}}"]\arrow[d, "F'\cur{\del_{0,1}}"] &  & F'\cur{I_{1,1}}\arrow[rr, "F'\cur{d_{1,1}}"]\arrow[d, "F'\cur{\del_{1,1}}"] &  & F'\cur{I_{2,1}}\arrow[rr, "F'\cur{d_{2,1}}"]\arrow[d, "F'\cur{\del_{2,1}}"] &  & \cdots \\
        0\arrow[r] & F'\cur{I_{0,2}}\arrow[rr, "F'\cur{d_{0,2}}"]\arrow[d, "F'\cur{\del_{0,2}}"] &  & F'\cur{I_{1,2}}\arrow[rr, "F'\cur{d_{1,2}}"]\arrow[d, "F'\cur{\del_{1,2}}"] &  & F'\cur{I_{2,2}}\arrow[rr, "F'\cur{d_{2,2}}"]\arrow[d, "F'\cur{\del_{2,2}}"] &  & \cdots \\
         & \vdots &  & \vdots &  & \vdots &  & 
    \end{tikzcd}
    }.
    \]
    Let $\Tot\cur{F'\cur{\mc{I}}}$ be the associated single complex. We now consider two possible spectral sequences, which we will denote by $^IE_{r}^{i,j}$ and $^{II}E_{r}^{i,j}$.

    To construct $^IE_{r}^{i,j}$, we first reorient the $^IE_{0}$ page to the following:
    \[
    \adjustbox{scale=1}{%
    \begin{tikzcd}
        & \vdots &  & \vdots &  & \vdots &  & \\
        0\arrow[r] & F'\cur{I_{0,2}}\arrow[rr, "F'\cur{d_{0,2}}"]\arrow[u, "F'\cur{\del_{0,2}}"] &  & F'\cur{I_{1,2}}\arrow[rr, "F'\cur{d_{1,2}}"]\arrow[u, "F'\cur{\del_{1,2}}"] &  & F'\cur{I_{2,2}}\arrow[rr, "F'\cur{d_{2,2}}"]\arrow[u, "F'\cur{\del_{2,2}}"] &  & \cdots \\
        0\arrow[r] & F'\cur{I_{0,1}}\arrow[rr, "F'\cur{d_{0,1}}"]\arrow[u, "F'\cur{\del_{0,1}}"] &  & F'\cur{I_{1,1}}\arrow[rr, "F'\cur{d_{1,1}}"]\arrow[u, "F'\cur{\del_{1,1}}"] &  & F'\cur{I_{2,1}}\arrow[rr, "F'\cur{d_{2,1}}"]\arrow[u, "F'\cur{\del_{2,1}}"] &  & \cdots \\
        0\arrow[r] & F'\cur{I_{0,0}}\arrow[rr, "F'\cur{d_{0,0}}"]\arrow[u, "F'\cur{\del_{0,0}}"] &  & F'\cur{I_{1,0}}\arrow[rr, "F'\cur{d_{1,0}}"]\arrow[u, "F'\cur{\del_{1,0}}"] &  & F'\cur{I_{2,0}}\arrow[rr, "F'\cur{d_{2,0}}"]\arrow[u, "F'\cur{\del_{2,0}}"] &  & \cdots \\
        & 0\arrow[u] &  & 0\arrow[u] &  & 0\arrow[u] &  & 
    \end{tikzcd}
    }.
    \]
    Now, for a fixed $i$, there exists a $(G',H')$-injective resolution, $\mc{I}_{i}$,
    \[
    \adjustbox{scale=1}{%
    \begin{tikzcd}
        0\arrow[r] & F\cur{I_{i}}\arrow[r, "\del_{i,-1}"] & I_{i,0}\arrow[r, "\del_{i,0}"] & I_{i,1}\arrow[r, "\del_{i,1}"] & \cdots.
    \end{tikzcd}
    }
    \]
    Then, by definition of $\textup{R}_{(G',H')}^{j}F'\cur{F\cur{I_{i}}}$,
    \[
    \adjustbox{scale=1}{%
    \begin{tikzcd}
        0\arrow[r] & F'\cur{I_{i,0}}\arrow[r, "F'\cur{\del_{i,0}}"] & F'\cur{I_{i,1}}\arrow[r, "F'\cur{\del_{i,1}}"] & F'\cur{I_{i,2}}\arrow[r, "F'\cur{\del_{i,2}}"] & \cdots
    \end{tikzcd}
    }
    \]
    is a complex whose homology is $\textup{R}_{(G',H')}^{j}F'\cur{F\cur{I_{i}}}$. Hence, the $^IE_{1}$ page looks like
    \[
    \adjustbox{scale=0.98}{%
    \begin{tikzcd}
        & \vdots & \vdots & \vdots & \\
        0\arrow[r] & \textup{R}_{(G',H')}^{2}F'\cur{F\cur{I_{0}}}\arrow[r]\arrow[u] & \textup{R}_{(G',H')}^{2}F'\cur{F\cur{I_{1}}}\arrow[r]\arrow[u] & \textup{R}_{(G',H')}^{2}F'\cur{F\cur{I_{2}}}\arrow[r]\arrow[u] & \cdots \\
        0\arrow[r] & \textup{R}_{(G',H')}^{1}F'\cur{F\cur{I_{0}}}\arrow[r]\arrow[u] & \textup{R}_{(G',H')}^{1}F'\cur{F\cur{I_{1}}}\arrow[r]\arrow[u] & \textup{R}_{(G',H')}^{1}F'\cur{F\cur{I_{2}}}\arrow[r]\arrow[u] & \cdots \\
        0\arrow[r] & \textup{R}_{(G',H')}^{0}F'\cur{F\cur{I_{0}}}\arrow[r]\arrow[u] & \textup{R}_{(G',H')}^{0}F'\cur{F\cur{I_{1}}}\arrow[r]\arrow[u] & \textup{R}_{(G',H')}^{0}F'\cur{F\cur{I_{2}}}\arrow[r]\arrow[u] & \cdots \\
        & 0\arrow[u] & 0\arrow[u] & 0\arrow[u] & 
    \end{tikzcd}
    }.
    \]

    Recall that $F\cur{I_{i}}$ is $(G',H')$-acyclic to $F'$, so $\textup{R}_{(G',H')}^{j}F'\cur{F\cur{I_{i}}}=0$ for $j>0$. Also, by \Cref{rf0}, $\textup{R}_{(G',H')}^{0}F'\cur{F\cur{I_{i}}}=F'\cur{F\cur{I_{i}}}$. Hence the non-zero terms are on the $i$-axis, so $^IE_{1}$ looks like
    \[
    \adjustbox{scale=1}{%
    \begin{tikzcd}
        0\arrow[r] & F'\cur{F\cur{I_{0}}}\arrow[rr, "F'\cur{F\cur{d_{0}}}"] &  & F'\cur{F\cur{I_{1}}}\arrow[rr, "F'\cur{F\cur{d_{1}}}"] &  & F'\cur{F\cur{I_{2}}}\arrow[rr, "F'\cur{F\cur{d_{2}}}"] &  & \cdots.
    \end{tikzcd}
    }
    \]
    Now taking homology of this complex, we get $$^IE_{2}^{i,j}=
    \begin{cases}
        \cur{\textup{R}_{(G,H)}^{i}\cur{F'\circ F}}\cur{M} & \textup{if }j=0 \\
        0 & \textup{if }j>0
    \end{cases}.$$
    Therefore, since $^IE_{2}^{i,j}$ collapses on the $i$-axis, $$H^{n}\cur{\Tot\cur{F'\cur{\mc{I}}}}\cong\cur{\textup{R}_{(G,H)}^{n}\cur{F'\circ F}}\cur{M}.$$
    
    Now we will construct $^{II}E_{r}^{i,j}$. To start we reorient the $^{II}E_{0}$ page to the following:
    \[
    \adjustbox{scale=1}{%
    \begin{tikzcd}
        & \vdots &  & \vdots &  & \vdots &  & \\
        0\arrow[r] & F'\cur{I_{2,0}}\arrow[rr, "F'\cur{\del_{2,0}}"]\arrow[u, "F'\cur{d_{2,0}}"] &  & F'\cur{I_{2,1}}\arrow[rr, "F'\cur{\del_{2,1}}"]\arrow[u, "F'\cur{d_{2,1}}"] &  & F'\cur{I_{2,2}}\arrow[rr, "F'\cur{\del_{2,2}}"]\arrow[u, "dF'\cur{_{2,2}}"] &  & \cdots \\
        0\arrow[r] & F'\cur{I_{1,0}}\arrow[rr, "F'\cur{\del_{1,0}}"]\arrow[u, "F'\cur{d_{1,0}}"] &  & F'\cur{I_{1,1}}\arrow[rr, "F'\cur{\del_{1,1}}"]\arrow[u, "F'\cur{d_{1,1}}"] &  & F'\cur{I_{1,2}}\arrow[rr, "F'\cur{\del_{1,2}}"]\arrow[u, "F'\cur{d_{1,2}}"] &  & \cdots \\
        0\arrow[r] & F'\cur{I_{0,0}}\arrow[rr, "F'\cur{\del_{0,0}}"]\arrow[u, "F'\cur{d_{0,0}}"] &  & F'\cur{I_{0,1}}\arrow[rr, "F'\cur{\del_{0,1}}"]\arrow[u, "F'\cur{d_{0,1}}"] &  & F'\cur{I_{0,2}}\arrow[rr, "F'\cur{\del_{0,2}}"]\arrow[u, "F'\cur{d_{0,2}}"] &  & \cdots \\
        & 0\arrow[u] &  & 0\arrow[u] &  & 0\arrow[u] &  & 
    \end{tikzcd}
    }.
    \]

    First, from our assumptions, the sequences
    \begin{equation*}
    \adjustbox{scale=1}{
    \begin{tikzcd}
        0\arrow[r] & \Ker\cur{F\cur{d_{j}}}\arrow[r] & F\cur{I_{j}}\arrow[r] & \IM\cur{F\cur{d_{j}}}\arrow[r] & 0
    \end{tikzcd}
    }
\end{equation*}
and
\begin{equation*}
    \adjustbox{scale=1}{
    \begin{tikzcd}
        0\arrow[r] & \IM\cur{F\cur{d_{j}}}\arrow[r] & \Ker\cur{F\cur{d_{j+1}}}\arrow[r] & H^{j+1}\cur{F\cur{I_{\bigcdot}}}\arrow[r] & 0
    \end{tikzcd}
    }
\end{equation*}
are $(G',H')$-exact sequences for $j\geq0$. By \Cref{cir}, we can take $(G',H')$-bar resolutions of both to get
    \[
    \adjustbox{scale=1}{%
    \begin{tikzcd}
        0\arrow[r] & \Ker\cur{F\cur{d_{j}}}\arrow[r]\arrow[d] & F\cur{I_{j}}\arrow[r]\arrow[d] & \IM\cur{F\cur{d_{j}}}\arrow[r]\arrow[d] & 0 \\
        0\arrow[r] & Z_{j,0}\arrow[r]\arrow[d] & I_{j,0}\arrow[r]\arrow[d] & B_{j+1,0}\arrow[r]\arrow[d] & 0 \\
        0\arrow[r] & Z_{j,1}\arrow[r]\arrow[d] & I_{j,1}\arrow[r]\arrow[d] & B_{j+1,1}\arrow[r]\arrow[d] & 0 \\
         & \vdots & \vdots & \vdots & 
    \end{tikzcd}
    },
    \]
    where
    \[
    \adjustbox{scale=1}{%
    \begin{tikzcd}
        0\arrow[r] & Z_{j,i}\arrow[r] & I_{j,i}\arrow[r] & B_{j+1,i}\arrow[r] & 0 
    \end{tikzcd}
    }
    \]
    splits for all $j\geq0$, and
    \[
    \adjustbox{scale=1}{%
    \begin{tikzcd}
        0\arrow[r] & \IM\cur{F\cur{d_{j}}}\arrow[r]\arrow[d] & \Ker\cur{F\cur{d_{j+1}}}\arrow[r]\arrow[d] & H^{j+1}\cur{F\cur{I_{\bigcdot}}}\arrow[r]\arrow[d] & 0 \\
        0\arrow[r] & B_{j+1,0}\arrow[r]\arrow[d] & Z_{j+1,0}\arrow[r]\arrow[d] & H_{j+1,0}\arrow[r]\arrow[d] & 0 \\
        0\arrow[r] & B_{j+1,1}\arrow[r]\arrow[d] & Z_{j+1,1}\arrow[r]\arrow[d] & H_{j+1,1}\arrow[r]\arrow[d] & 0 \\
         & \vdots & \vdots & \vdots & 
    \end{tikzcd}
    },
    \]
    where
    \[
    \adjustbox{scale=1}{%
    \begin{tikzcd}
        0\arrow[r] & B_{j+1,i}\arrow[r] & Z_{j+1,i}\arrow[r] & H_{j+1,i}\arrow[r] & 0 
    \end{tikzcd}
    }
    \]
    splits for all $j\geq0$. Note that $Z_{j,i}=\Ker\cur{d_{j,i}}$, $B_{j+1,i}=\IM\cur{d_{j,i}}$, and $H_{j+1,i}=H^{j+1}\cur{I_{\bigcdot,i}}=\Ker\cur{d_{j+1,i}}/\IM\cur{d_{j,i}}$

    Define $$\textup{}H_d^{j,i}\cur{F'\cur{\mc{I}}}:=\Ker\cur{F'\cur{d_{j,i}}}\Big/\IM\cur{F'\cur{d_{j-1,i}}}$$ and similarly $$\textup{}H_d^{j,i}\cur{\mc{I}}:=\Ker\cur{d_{j,i}}\Big/\IM\cur{d_{j-1,i}}.$$
    Now, let $I_{(i)}$ be the $i$-th row of the double complex $\mc{I}=\{I_{j,i}\}$ (which will correspond to the $i$-th column of $^{II}E_{0}$). Then, for $j>0$,
    \begin{align*}
        H_d^{j,i}\cur{F'\cur{\mc{I}}}\cong&H^{j}\cur{F'\cur{I_{(i)}}}\\
        \cong&\Ker\cur{F'\cur{d_{j,i}}}\Big/\IM\cur{F'\cur{d_{j-1,i}}}\\
        \cong&F'\cur{\Ker\cur{d_{j,i}}}\Big/F'\cur{\IM\cur{d_{j-1,i}}}\\
        \cong&F'\cur{\textup{}H_d^{j,i}\cur{\mc{I}}},
    \end{align*}
    where the third line comes from the fact that
    \[
    \adjustbox{scale=1}{%
    \begin{tikzcd}
        0\arrow[r] & \Ker\cur{d_{j,i}}\arrow[r] & I_{j,i}\arrow[r] & \IM\cur{d_{j,i}}\arrow[r] & 0 
    \end{tikzcd}
    }
    \]
    is split with each term $(G',H')$-injective, and the fourth line comes from the fact that
    \[
    \adjustbox{scale=1}{%
    \begin{tikzcd}
        0\arrow[r] & \IM\cur{d_{j-1,i}}\arrow[r] & \Ker\cur{d_{j,i}}\arrow[r] & H^{j}\cur{I_{\bigcdot,i}}\arrow[r] & 0 
    \end{tikzcd}
    }
    \]
    is split with each term $(G',H')$-injective.

    If $j=0$, then $$H^{0}\cur{F'\cur{I_{(i)}}}\cong \Ker\cur{F'\cur{d_{0,i}}}\cong F'\cur{\Ker\cur{d_{0,i}}}\cong F'\cur{\textup{}H_d^{0,i}\cur{\mc{I}}}.$$

    Therefore, the $^{II}E_{1}$ page looks like
    \[
    \adjustbox{scale=1}{%
    \begin{tikzcd}
        & \vdots & \vdots & \vdots & \\
        0\arrow[r] & F'\cur{\textup{}H_d^{2,0}\cur{\mc{I}}}\arrow[r]\arrow[u] & F'\cur{\textup{}H_d^{2,1}\cur{\mc{I}}}\arrow[r]\arrow[u] & F'\cur{\textup{}H_d^{2,2}\cur{\mc{I}}}\arrow[r]\arrow[u] & \cdots \\
        0\arrow[r] & F'\cur{\textup{}H_d^{1,0}\cur{\mc{I}}}\arrow[r]\arrow[u] & F'\cur{\textup{}H_d^{1,1}\cur{\mc{I}}}\arrow[r]\arrow[u] & F'\cur{\textup{}H_d^{1,2}\cur{\mc{I}}}\arrow[r]\arrow[u] & \cdots \\
        0\arrow[r] & F'\cur{\textup{}H_d^{0,0}\cur{\mc{I}}}\arrow[r]\arrow[u] & F'\cur{\textup{}H_d^{0,1}\cur{\mc{I}}}\arrow[r]\arrow[u] & F'\cur{\textup{}H_d^{0,2}\cur{\mc{I}}}\arrow[r]\arrow[u] & \cdots \\
        & 0\arrow[u] & 0\arrow[u] & 0\arrow[u] & 
    \end{tikzcd}
    }.
    \]
    
    Now (just focusing on a single row of the $^{II}E_{1}$ page), since the complex $H_d^{j,i}\cur{\mc{I}}$ with $i\geq0$ is a $(G',H')$-injective resolution of $$H^{j}\cur{F\cur{I_{\bigcdot}}}\cong\textup{R}_{(G,H)}^{j}F\cur{M},$$ we have that $$H^{i}\cur{F'\cur{\textup{}H_d^{j,i}\cur{\mc{I}}}}\cong\textup{R}_{(G',H')}^{i}F'\cur{\textup{R}_{(G,H)}^{j}F\cur{M}},$$ 
    and so $$^{II}E_{2}^{i,j}=\textup{R}_{(G',H')}^{i}F'\cur{\textup{R}_{(G,H)}^{j}F\cur{M}}.$$
    
    Therefore, since both $^IE_{2}^{i,j}$ and $^{II}E_{2}^{i,j}$ converge to $H^{i+j}\cur{\Tot\cur{F'\cur{\mc{I}}}}$, $$\textup{R}_{(G',H')}^{i}F'\cur{\textup{R}_{(G,H)}^{j}F\cur{M}}\Rightarrow\cur{\textup{R}_{(G,H)}^{i+j}\cur{F'\circ F}}\cur{M}.$$
\end{proof}

In \cite[Proposition 5.2.1 and Theorem 6.3.2]{LNW25}, the authors apply \Cref{ggspfrrdf} to Category $\mathcal{O}$.

Note that because of \Cref{rif0}, showing that $F$ sends $(G,H)$-injective modules to $(G',H')$-injective modules is sufficient to get a spectral sequence.

We now want to find an easier condition to check for a functor to be $(H,H')$-split, so we give the following result which gives a sufficient condition.

\begin{proposition}\label{cfftbrs}
    Given any $(G,H)$-exact sequence
    \[
    \adjustbox{scale=1}{%
    \begin{tikzcd}
        0\arrow[r] & M_{1}\arrow[r, "d_{1}"] & M_{2}\arrow[r, "d_{2}"] & M_{3}\arrow[r] & 0, 
    \end{tikzcd}
    }
    \]
    if
    \[
    \adjustbox{scale=1}{%
    \begin{tikzcd}
        0\arrow[r] & F\cur{M_{1}}\arrow[r] & F\cur{M_{2}}\arrow[r] & F\cur{M_{3}}
    \end{tikzcd}
    }
    \]
    is $(G',H')$-exact, then $F$ is $(H,H')$-split.
\end{proposition}

\begin{proof}
    Let
    \[
    \adjustbox{scale=1}{%
    \begin{tikzcd}
        \cdots\arrow[r] & M_{i-1}\arrow[r, "d_{i-1}"] & M_{i}\arrow[r, "d_{i}"] & M_{i+1}\arrow[r] & \cdots 
    \end{tikzcd}
    }
    \]
    be a $(G,H)$-exact sequence and apply $F$ to get
    \[
    \adjustbox{scale=1}{%
    \begin{tikzcd}
        \cdots\arrow[r] & F\cur{M_{i-1}}\arrow[r, "F\cur{d_{i-1}}"] & F\cur{M_{i}}\arrow[r, "F\cur{d_{i}}"] & F\cur{M_{i+1}}\arrow[r] & \cdots. 
    \end{tikzcd}
    }
    \]
    Then the sequences
    \begin{equation}\label{eq:9}
        \adjustbox{scale=1}{%
    \begin{tikzcd}
        0\arrow[r] & \Ker\cur{F\cur{d_{i}}}\arrow[r, "i"] & F\cur{M_{i}}\arrow[r, "F\cur{d_{i}}"] & \IM\cur{F\cur{d_{i}}}\arrow[r] & 0 
    \end{tikzcd}
    }
    \end{equation}
    and
    \begin{equation}\label{eq:10}
        \adjustbox{scale=1}{%
    \begin{tikzcd}
        0\arrow[r] & \IM\cur{F\cur{d_{i-1}}}\arrow[r, "i"] & \Ker\cur{F\cur{d_{i}}}\arrow[r, "\pi"] & H^{i}\cur{M_{\bigcdot}}\arrow[r] & 0 
    \end{tikzcd}
    }
    \end{equation}
    are exact, so we need to show they split as $H'$-modules.

    From our assumption, the following sequence is $(G',H')$-exact
    \[
    \adjustbox{scale=1}{%
    \begin{tikzcd}
        0\arrow[r] & F\cur{\Ker\cur{d_{i}}}\arrow[r, "F(i)"] & F\cur{M_{i}}\arrow[r, "F\cur{d_{i}}"] & F\cur{\IM\cur{d_{i}}},
    \end{tikzcd}
    }
    \]
    so $\IM\cur{F\cur{d_{i}}}$ is an $H'$ direct summand of $F\cur{M_{i}}$ and hence \sref{eq:9} is $(G',H')$-exact. Now, $\IM\cur{F\cur{d_{i-1}}}$ is an $H'$ direct summand of $F\cur{\IM\cur{d_{i-1}}}\cong F\cur{\Ker\cur{d_{i}}}\cong\Ker\cur{F\cur{d_{i}}}$, hence \sref{eq:10} is $(G',H')$-exact. Therefore, $F$ is $(H,H')$-split.
\end{proof}
\subsection{}

We now apply \Cref{ggspfrrdf} and \Cref{cfftbrs} to get a relative generalized Frobenius reciprocity result, although for now we will need some strong conditions. Later on, we will give more cases when relative generalized Frobenius reciprocity holds.

\begin{corollary}\label{reisshkss}
    Let $H$ and $K$ be closed subgroups of $G$ such that $\Mod{H}$ or $\Mod{K}$ is semisimple, then there exists a spectral sequence such that $$E_2^{i,j}=\Ext{(G,K)}{i}\cur{M,\rRind{(H,H\cap K)}{j}{H}{G}(N)}\Rightarrow\Ext{(H,H\cap K)}{i+j}\cur{M,N}.$$ Furthermore, if $\Mod{H}$ is semisimple, then the spectral sequence collapses and $$\Ext{(G,K)}{i}\cur{M,\ind{H}{G}(N)}\cong\Ext{(H,H\cap K)}{i}\cur{M,N}\text{ for }i\geq0.$$
\end{corollary}

\begin{proof}
    First, by \Cref{iriri} $\ind{H}{G}$ takes $(H,H\cap K)$-injective modules to $(G,K)$-injective modules. So in either case we need to show $\ind{H}{G}$ is $(H\cap K,K)$-split. Let
    \[
    \adjustbox{scale=1}{%
    \begin{tikzcd}
        0\arrow[r] & N_{1}\arrow[r] & N\arrow[r] & N_{2}\arrow[r] & 0 
    \end{tikzcd}
    }
    \]
    be an $(H,H\cap K)$-exact sequence.
    
    Assume $\Mod{K}$ is semisimple, then it is clear that
    \[
    \adjustbox{scale=1}{%
    \begin{tikzcd}
        0\arrow[r] & \ind{H}{G}\cur{N_{1}}\arrow[r] & \ind{H}{G}\cur{N}\arrow[r] & \ind{H}{G}\cur{N_{2}} & 
    \end{tikzcd}
    }
    \]
    is $(G,K)$-exact, so from \Cref{cfftbrs} $\ind{H}{G}$ is $(H\cap K,K)$-split.

    Assume $\Mod{H}$ is semisimple, then
    \[
    \adjustbox{scale=1}{%
    \begin{tikzcd}
        0\arrow[r] & N_{1}\arrow[r] & N\arrow[r] & N_{2}\arrow[r] & 0 
    \end{tikzcd}
    }
    \]
    is split, and so
    \[
    \adjustbox{scale=1}{%
    \begin{tikzcd}
        0\arrow[r] & \ind{H}{G}\cur{N_{1}}\arrow[r] & \ind{H}{G}\cur{N}\arrow[r] & \ind{H}{G}\cur{N_{2}}\arrow[r] & 0
    \end{tikzcd}
    }
    \]
    is also split and hence $(G,K)$-exact. Therefore, since $\ind{H}{G}$ takes $(H,H\cap K)$-exact sequences to $(G,K)$-exact sequences, the spectral sequence will collapses and so $$\Ext{(G,K)}{i}\cur{M,\ind{H}{G}(N)}\cong\Ext{(H,H\cap K)}{i}\cur{M,N}.$$
\end{proof}

If we take $K=T$ a torus, then it is known that $\Mod{T}$ is semisimple, and so one has:

\begin{corollary}
    Let $H$ and $T$ be closed subgroups of $G$ such that $T$ is a torus, then there exists a spectral sequence such that $$E_2^{i,j}=\Ext{(G,T)}{i}\cur{M,\rRind{(H,H\cap T)}{j}{H}{G}(N)}\Rightarrow\Ext{(H,H\cap T)}{i+j}\cur{M,N}.$$
\end{corollary}

Similarly, we also apply \Cref{ggspfrrdf} and \Cref{cfftbrs} to get a relative version of \cite[I.4.5(c)]{Jan03}, which again for now we will need some strong conditions, but later on we will give more cases when the result holds.

\begin{corollary}\label{riisshkss}
    Let $H'$ be a closed subgroup of $G$, and $H$ and $K$ be closed subgroups of $H'$ such that $\Mod{H}$ or $\Mod{K}$ is semisimple, then there exists a spectral sequence such that $$E_2^{i,j}=\rRind{(H',K)}{i}{H'}{G}\cur{\rRind{(H,H\cap K)}{j}{H}{H'}(M)}\Rightarrow\rRind{(H,H\cap K)}{i+j}{H}{G}\cur{M}.$$ Furthermore, if $\Mod{H}$ is semisimple, then the spectral sequence collapses and $$\rRind{(H',K)}{i}{H'}{G}\cur{\ind{H}{H'}(M)}\cong\rRind{(H,H\cap K)}{i}{H}{G}\cur{M}\text{ for }i\geq0.$$
\end{corollary}

\begin{proof}
    First, by \Cref{iriri} $\ind{H}{H'}$ takes $(H,H\cap K)$-injective modules to $(H',K)$-injective modules. So in either case we need to show $\ind{H}{H'}$ is $(H\cap K,K)$-split. Let
    \[
    \adjustbox{scale=1}{%
    \begin{tikzcd}
        0\arrow[r] & M_{1}\arrow[r] & M\arrow[r] & M_{2}\arrow[r] & 0 
    \end{tikzcd}
    }
    \]
    be an $(H,H\cap K)$-exact sequence.
    
    Assume $\Mod{K}$ is semisimple, then it is clear that
    \[
    \adjustbox{scale=1}{%
    \begin{tikzcd}
        0\arrow[r] & \ind{H}{H'}\cur{M_{1}}\arrow[r] & \ind{H}{H'}\cur{M}\arrow[r] & \ind{H}{H'}\cur{M_{2}} & 
    \end{tikzcd}
    }
    \]
    is $(H',K)$-exact, so from \Cref{cfftbrs} $\ind{H}{H'}$ is $(H\cap K,K)$-split.

    Assume $\Mod{H}$ is semisimple, then
    \[
    \adjustbox{scale=1}{%
    \begin{tikzcd}
        0\arrow[r] & M_{1}\arrow[r] & M\arrow[r] & M_{2}\arrow[r] & 0 
    \end{tikzcd}
    }
    \]
    is split, and so
    \[
    \adjustbox{scale=1}{%
    \begin{tikzcd}
        0\arrow[r] & \ind{H}{H'}\cur{M_{1}}\arrow[r] & \ind{H}{H'}\cur{M}\arrow[r] & \ind{H}{H'}\cur{M_{2}}\arrow[r] & 0
    \end{tikzcd}
    }
    \]
    is also split and hence $(H',K)$-exact. Therefore, since $\ind{H}{H'}$ takes $(H,H\cap K)$-exact sequences to $(H',K)$-exact sequences, the spectral sequence will collapses and so $$\rRind{(H',K)}{i}{H'}{G}\cur{\ind{H}{H'}(M)}\cong\rRind{(H,H\cap K)}{i}{H}{G}\cur{M}.$$
\end{proof}
\subsection{}

We are almost ready to apply our general theory of relative right derived functors to get some specific results. But first we will need the following two results.

\begin{proposition}\label{crrdf1}
    Let $G$, $G'$, and $G''$ be algebraic groups, $H$ a closed subgroup of $G$, $H'$ a closed subgroup of $G'$, and $F:\Mod{G}\to\Mod{G'}$ and $F':\Mod{G'}\to\Mod{G''}$ be additive left exact covariant functors. 
    
    If $F'$ is exact, then for all $G$-modules $M$, $$\cur{\textup{R}_{(G,H)}^{i}\cur{F'\circ F}}\cur{M}\cong\cur{F'\circ\textup{R}_{(G,H)}^{i}F}\cur{M}\text{ for }i\geq0.$$
\end{proposition}

\begin{proof}
    Consider a $(G,H)$-injective resolution of $M$,
    \[
    \adjustbox{scale=1}{%
    \begin{tikzcd}
        0\arrow[r] & M\arrow[r, "d_{-1}"] & I_{0}\arrow[r, "d_{0}"] & I_{1}\arrow[r, "d_{1}"] & \cdots, 
    \end{tikzcd}
    }
    \]
    and apply $F$ to get 
    \[
    \adjustbox{scale=1}{%
    \begin{tikzcd}
        0\arrow[r] & F\cur{M}\arrow[r, "F\cur{d_{-1}}"] & F\cur{I_{0}}\arrow[r, "F\cur{d_{0}}"] & F\cur{I_{1}}\arrow[r, "F\cur{d_{1}}"] & \cdots. 
    \end{tikzcd}
    }
    \]
    One has that $\textup{R}_{(G,H)}^{i}F\cur{M}\cong \Ker\cur{F\cur{d_{i}}}/\IM\cur{F\cur{d_{i-1}}}$, so since $F'$ is exact,
    \begin{align*}
        F'\cur{\textup{R}_{(G,H)}^{i}F\cur{M}}&\cong F'\cur{\Ker\cur{F\cur{d_{i}}}/\IM\cur{F\cur{d_{i-1}}}}\\
        &\cong F'\cur{\Ker\cur{F\cur{d_{i}}}}/F'\cur{\IM\cur{F\cur{d_{i-1}}}}\\
        &\cong \Ker\cur{F'\cur{F\cur{d_{i}}}}/\IM\cur{F'\cur{F\cur{d_{i-1}}}}\\
        &\cong \textup{R}_{(G,H)}^{i}\cur{F'\circ F}\cur{M}.
    \end{align*}
\end{proof}

\begin{proposition}\label{crrdf2}
    Let $G$, $G'$, and $G''$ be algebraic groups, $H$ a closed subgroup of $G$, $H'$ a closed subgroup of $G'$, and $F:\Mod{G}\to\Mod{G'}$ and $F':\Mod{G'}\to\Mod{G''}$ be additive left exact covariant functors. 
    
    If $F$ maps $(G,H)$-injective objects to $(G',H')$-acyclic objects of $F'$, and sends $(G,H)$-exact sequences to $(G',H')$-exact sequences, then for all $G$-modules $M$, $$\cur{\textup{R}_{(G,H)}^{i}\cur{F'\circ F}}\cur{M}\cong\cur{\textup{R}_{(G',H')}^{i}F'\circ F}\cur{M}\text{ for }i\geq0.$$
\end{proposition}

\begin{proof}
    Under the assumptions, it is clear that $F$ is $(H,H')$-split, so \Cref{ggspfrrdf} applies and there is a spectral sequence
    $$\textup{R}_{(G',H')}^{i}F'\cur{\textup{R}_{(G,H)}^{j}F\cur{M}}\Rightarrow\cur{\textup{R}_{(G,H)}^{i+j}\cur{F'\circ F}}\cur{M}.$$
    Now, consider a $(G,H)$-injective resolution,
    \[
    \adjustbox{scale=1}{%
    \begin{tikzcd}
        0\arrow[r] & M\arrow[r, "d_{-1}"] & I_{0}\arrow[r, "d_{0}"] & I_{1}\arrow[r, "d_{1}"] & \cdots,
    \end{tikzcd}
    }
    \]
    and apply $F$ to get
    \begin{equation}\label{eq:11}
        \adjustbox{scale=1}{%
    \begin{tikzcd}
        0\arrow[r] & F\cur{M}\arrow[r, "F\cur{d_{-1}}"] & F\cur{I_{0}}\arrow[r, "F\cur{d_{0}}"] & F\cur{I_{1}}\arrow[r, "F\cur{d_{1}}"] & \cdots.
    \end{tikzcd}
    }
    \end{equation}
    From our assumptions, \sref{eq:11} is $(G',H')$-exact, and so the cohomology of it is always zero. Hence, $$\textup{R}_{(G,H)}^{j}F\cur{M}\cong\begin{cases}
        F\cur{M} & \text{ if }j=0\\
        0 & \text{ if }j>0
    \end{cases},$$ and so the spectral collapses and the result follows.
\end{proof}
\subsection{}

When looking at the normal cohomology, it is well known that for $G$-modules $M$, $N$, and $V$, such that $V$ is finite dimensional, and $V^*=\Hom{G}{V,k}$ is the dual module of $V$, that $$\Ext{G}{n}\cur{M,V\otimes N}\cong\Ext{G}{n}\cur{M\otimes V^*,N}\text{ for }n\geq0,$$ see \cite[II.4.4]{Jan03}. We now prove a relative version of this, which will be important later on for finding another condition for when $\ind{H}{G}$ takes relative exact sequences to relative exact sequences.

\begin{lemma}\label{mfdmie}
    Let $M$, $N$, and $V$ be $G$-modules. If $V$ is finitely generated and projective as a $k$-module, then 
    \begin{enumerate}[(a)]
        \item $\displaystyle \Ext{(G,H)}{n}\cur{M,V\otimes N}\cong\Ext{(G,H)}{n}\cur{M\otimes V^*,N}$ for $n\geq0$, and
        \item\label{mfdmieb} $\displaystyle \Ext{(G,H)}{n}\cur{V,N}\cong\Ext{(G,H)}{n}\cur{k,V^*\otimes N}$ for $n\geq0$.
    \end{enumerate}
\end{lemma}

\begin{proof}
    By \Cref{triri}, $V\otimes-$ takes $(G,H)$-injective modules to $(G,H)$-injective modules. Since $V$ is projective as a $k$-module, $V\otimes-$ is exact, so by \Cref{tetretre}, sends $(G,H)$-exact sequences to $(G,H)$-exact sequences. Hence, by \Cref{crrdf2},
    \begin{align*}
        \Ext{(G,H)}{n}\cur{M,V\otimes N}&\cong\cur{\textup{R}_{(G,H)}^{n}\Hom{G}{M,-}\circ\cur{V\otimes -}}(N)\\
        &\cong\textup{R}_{(G,H)}^{n}\cur{\Hom{G}{M,-}\circ\cur{V\otimes -}}(N)\\
        &\cong\textup{R}_{(G,H)}^{n}\cur{\Hom{G}{M\otimes V^*,-}}(N)\\
        &\cong\Ext{(G,H)}{n}\cur{M\otimes V^*,N}.
    \end{align*}
    This proves part (a), taking the case when $M=k$ yields part (b). 
\end{proof}

We remark that \Crefdefpart{mfdmie}{mfdmieb} was first proven in \textup{\cite[Proposition 2.3]{Kim65}}.
\subsection{}

We now prove that under the right conditions there exists a relative generalized Frobenius reciprocity result, which will be important for computing examples.

\begin{proposition}\label{rfr}
    Let $H$ and $K$ be closed subgroups of $G$ such that one of the following holds:
    \begin{enumerate}[(a)]
        \item $HK=G$, or
        \item $\Mod{H}$ is semisimple,
    \end{enumerate}
    then $\Ext{(G,K)}{i}\cur{M,\ind{H}{G}(N)}\cong\Ext{\cur{H,H\cap K}}{i}\cur{M,N}\text{ for }i\geq0.$
\end{proposition}

\begin{proof}
    If $\Mod{H}$ is semisimple, one can apply \Cref{reisshkss} to get the result.
    
    If $HK=G$, by \Cref{iriri} and \Cref{iores}, $\ind{H}{G}$ takes $\cur{H,H\cap K}$-injective modules to $\cur{G,K}$-injective modules, and $\cur{H,H\cap K}$-exact sequences to $\cur{G,K}$-exact sequences, so apply \Cref{crrdf2} to get
    \begin{align*}
        \Ext{(G,K)}{i}\cur{M,\ind{H}{G}(N)}&\cong\cur{\textup{R}_{(G,K)}^{i}\Hom{G}{M,-}\circ\ind{H}{G}(-)}(N)\\
        &\cong\textup{R}_{\cur{H,H\cap K}}^{i}\cur{\Hom{G}{M,-}\circ\ind{H}{G}(-)}(N)\\
        &\cong\textup{R}_{\cur{H,H\cap K}}^{i}\cur{\Hom{H}{M,-}}(N)\\
        &\cong\Ext{\cur{H,H\cap K}}{i}\cur{M,N}.
    \end{align*}
\end{proof}

Note that by \Cref{rfr} $$\Ext{(G,H)}{n}\cur{k,\ind{\crly{1}}{G}(N)}\cong\Ext{\cur{\crly{1},\crly{1}}}{n}\cur{k,N}\cong\begin{cases}
    N & n=0\\
    0 & n>0
\end{cases}.$$

The next proposition gives a relative version of \cite[I.4.5(c)]{Jan03}, although due to the (fairly strong) conditions that are required, one gets an isomorphism, not just a spectral sequence.


\begin{proposition}\label{rrdisbsg}
    Let $H'$ be a closed subgroup of $G$ and $H$ and $K$ be closed subgroups of $H'$ such that one of the following holds:
    \begin{enumerate}[(a)]
        \item $HK=H'$, or
        \item $\Mod{H}$ is semisimple,
    \end{enumerate}
    then $\rRind{\cur{H',K}}{i}{H'}{G}\cur{\ind{H}{H'}(M)}\cong\rRind{\cur{H,H\cap K}}{i}{H}{G}(M)\text{ for }i\geq0.$
\end{proposition}

\begin{proof}
    If $\Mod{H}$ is semisimple, apply \Cref{riisshkss} to get the result.
    
    If $HK=H'$, by \Cref{iriri} and \Cref{iores}, $\ind{H}{H'}$ takes $\cur{H,H\cap K}$-injective modules to $\cur{H',K}$-injective modules, and $\cur{H,H\cap K}$-exact sequences to $\cur{H',K}$-exact sequences, so apply \Cref{crrdf2} to get
    \begin{align*}
        \rRind{(H',K)}{i}{H'}{G}\cur{\ind{H}{H'}(M)}&\cong\cur{\textup{R}_{(H',K)}^{i}\ind{H'}{G}(-)\circ\ind{H}{H'}(-)}(M)\\
        &\cong\textup{R}_{\cur{H,H\cap K}}^{i}\cur{\ind{H'}{G}(-)\circ\ind{H}{H'}(-)}(M)\\
        &\cong\textup{R}_{\cur{H,H\cap K}}^{i}\cur{\ind{H}{G}(-)}(M)\\
        &\cong\rRind{\cur{H,H\cap K}}{i}{H}{G}\cur{M}.
    \end{align*}
\end{proof}

Note that by \Cref{rrdisbsg} $$\rRind{(H,K)}{n}{H}{G}\cur{\ind{\crly{1}}{H}(M)}\cong\rRind{\cur{\crly{1},\crly{1}}}{n}{\crly{1}}{G}(M)\cong\begin{cases}
    \ind{\crly{1}}{G}(M) & n=0\\
    0 & n>0
\end{cases}.$$
\subsection{}

Another interesting result which can be obtained is a relative generalized tensor identity.

\begin{theorem}\label{rgti}
    Let $M$ be an $H$-module and $N$ be a $G$-module such that $N$ is a flat $k$-module. Then for all $i\geq0$ $$\rRind{(H,H\cap K)}{i}{H}{G}\cur{\res{H}{G}\cur{N}\otimes M}\cong N\otimes\rRind{(H,H\cap K)}{i}{H}{G}\cur{M}.$$
\end{theorem}

\begin{proof}
    First, $$\cur{\ind{H}{G}\cur{-}\circ\cur{\res{H}{G}\cur{N}\otimes-}}\cur{M}\cong\cur{\cur{N\otimes-}\circ\ind{H}{G}\cur{-}}\cur{M}.$$ Next, since $N$ is a flat $k$-module, $N\otimes-$ (resp. $\res{H}{G}\cur{N}\otimes-$) is exact, and by \Cref{tetretre}, $\res{H}{G}\cur{N}\otimes-$ sends $(H,H\cap K)$-exact sequences to $(H,H\cap K)$-exact sequences. Now, from \Cref{triri}, $\res{H}{G}\cur{N}\otimes-$ sends $(H,H\cap K)$-injective modules to $(H,H\cap K)$-injective modules. Therefore, apply \Cref{crrdf1} and \Cref{crrdf2} to get
    \begin{align*}
        \rRind{(H,H\cap K)}{i}{H}{G}\cur{\res{H}{G}\cur{N}\otimes M}&\cong\cur{\textup{R}_{(H,H\cap K)}^{i}\ind{H}{G}\cur{-}\circ\cur{\res{H}{G}\cur{N}\otimes-}}\cur{M}\\
        &\cong\textup{R}_{(H,H\cap K)}^{i}\cur{\ind{H}{G}\cur{-}\circ\cur{\res{H}{G}\cur{N}\otimes-}}\cur{M}\\
        &\cong\textup{R}_{(H,H\cap K)}^{i}\cur{\cur{N\otimes-}\circ\ind{H}{G}\cur{-}}\cur{M}\\
        &\cong\cur{\cur{\cur{N\otimes-}\circ\textup{R}_{(H,H\cap K)}^{i}\ind{H}{G}\cur{-}}}\cur{M}\\
        &\cong N\otimes\rRind{(H,H\cap K)}{i}{H}{G}\cur{M}.
    \end{align*}
\end{proof}
\subsection{Induction With Semi-Direct Products}

Let $G'$ be an algebraic group that acts on $G$ via automorphisms. We can then form the semi-direct product $G\rtimes G'$. If $G'$ stabilises a closed subgroup $H$ of $G$, then we can also form $H\rtimes G'$ and have an isomorphism of functors $$\res{G}{G\rtimes G'}\circ\ind{H\rtimes G'}{G\rtimes G'}\cong\ind{H}{G}\circ\res{H}{H\rtimes G'}.$$ See \cite[I.4.9]{Jan03} for the details, or use \Cref{riiir} since $G\cur{H\rtimes G'}=G\rtimes G'$ and $G\cap\cur{H\rtimes G'}=H$.

\begin{proposition}\label{rrisrirfsdp}
    Let $H$ and $K$ be closed subgroups of $G$ such that $K\leq H$ and $M$ be an $H\rtimes G'$-module, then $$\res{G}{G\rtimes G'}\cur{\rRind{\cur{H\rtimes G',K\rtimes G'}}{n}{H\rtimes G'}{G\rtimes G'}(M)}\cong\rRind{\cur{H,K}}{n}{H}{G}\cur{\res{H}{H\rtimes G'}(M)}\text{ for }n\geq0.$$
\end{proposition}

\begin{proof}
    First, it is clear that $\res{G}{G\rtimes G'}$ and $\res{H}{H\rtimes G'}$ are exact and that $\res{H}{H\rtimes G'}$ takes\\ $\cur{H\rtimes G',K\rtimes G'}$-exact sequences to $(H,K)$-exact sequences (since a $K\rtimes G'$-homomorphism is also a $K$-homomorphism).

    Next, let $M$ be $\cur{H\rtimes G',K\rtimes G'}$-injective, so by \Cref{riiffds} $M$ is a direct summand of some $\ind{K\rtimes G'}{H\rtimes G'}(I)$. Hence, $\res{H}{H\rtimes G'}(M)$ is a direct summand of $\res{H}{H\rtimes G'}\cur{\ind{K\rtimes G'}{H\rtimes G'}(I)}\cong\ind{K}{H}\cur{\res{K}{K\rtimes G'}(I)}$ which is $(H,K)$-injective by \Cref{iri}. Therefore, by \Cref{dsim}, $\res{H}{H\rtimes G'}(M)$ is $(H,K)$-injective.

    Now, apply \Cref{crrdf1} and \Cref{crrdf2} to get
    \begin{align*}
        \res{G}{G\rtimes G'}\cur{\rRind{\cur{H\rtimes G',K\rtimes G'}}{n}{H\rtimes G'}{G\rtimes G'}(M)}\cong&\cur{\res{G}{G\rtimes G'}(-)\circ\rRind{\cur{H\rtimes G',K\rtimes G'}}{n}{H\rtimes G'}{G\rtimes G'}(-)}(M)\\
        \cong&\textup{R}_{\cur{H\rtimes G',K\rtimes G'}}^{n}\cur{\res{G}{G\rtimes G'}(-)\circ\ind{H\rtimes G'}{G\rtimes G'}(-)}(M)\\
        \cong&\textup{R}_{\cur{H\rtimes G',K\rtimes G'}}^{n}\cur{\ind{H}{G}(-)\circ\res{H}{H\rtimes G'}(-)}(M)\\
        \cong&\cur{\rRind{\cur{H,K}}{n}{H}{G}(-)\circ\res{H}{H\rtimes G'}(-)}(M)\\
        \cong&\rRind{\cur{H,K}}{n}{H}{G}\cur{\res{H}{H\rtimes G'}(M)}.
    \end{align*}
\end{proof}

In fact, using the same argument from \Cref{rrisrirfsdp}, one obtains:

\begin{proposition}\label{rrisrirwcpsch}
    Let $H$, $H'$, and $K$ be closed subgroups of $G$ such that $H\leq H'$ and $M$ be an $H'$-module. If $\res{K}{G}\circ\ind{H'}{G}\cong\ind{H'\cap K}{G}\circ\res{H'\cap K}{H'}$ and $\res{H'\cap K}{H'}\circ\ind{H}{H'}\cong\ind{H\cap K}{H'\cap K}\circ\res{H\cap K}{H}$, then $$\res{K}{G}\cur{\rRind{\cur{H',H}}{n}{H'}{G}(M)}\cong\rRind{\cur{H'\cap K,H\cap K}}{n}{H'\cap K}{K}\cur{\res{H'\cap K}{H'}(M)}\text{ for }n\geq0.$$
\end{proposition}

Note, that the conditions for \Cref{rrisrirwcpsch} hold when $H'K=G$ and $H\cur{H'\cap K}=H'$.
\subsection{}

We now want to compare the relative cohomology of $H$ to the relative right derived functors of $\ind{H}{G}$.

\begin{proposition}\label{retkgri}
    Let $M$ be an $H$-module and $H$ and $K$ be closed subgroups of $G$ such that $K\leq H$, then $$\Ext{(H,K)}{n}\cur{k,M\otimes k[G]}\cong\rRind{(H,K)}{n}{H}{G}(M)\text{ for }n\geq0.$$
\end{proposition}

\begin{proof}
    Let $\mc{F}$ be the forgetful functor from $\Mod{G}$ to $\Mod{k}$, then from the definition of $\ind{H}{G}$ there exists an isomorphism $$\mc{F}\circ\ind{H}{G}\cong\Hom{H}{k,-}\circ\cur{-\otimes k[G]}.$$

    First, it is clear that both $\mc{F}$ and $\cur{-\otimes k[G]}$ are exact. Now, by \Cref{tetretre} and \Cref{triri}, $\cur{-\otimes k[G]}$ takes $(H,K)$-injective modules to $(H,K)$-injective modules and $(H,K)$-exact sequences to $(H,K)$-exact sequences. So apply \Cref{crrdf1} and \Cref{crrdf2} to get
    \begin{align*}
        \rRind{(H,K)}{n}{H}{G}(M)\cong&\mc{F}\circ\rRind{(H,K)}{n}{H}{G}(M)\\
        \cong&\cur{\mc{F}\circ\rRind{(H,K)}{n}{H}{G}(-)}(M)\\
        \cong&\textup{R}_{(H,K)}^{n}\cur{\mc{F}\circ\ind{H}{G}(-)}(M)\\
        \cong&\textup{R}_{(H,K)}^{n}\cur{\Hom{H}{k,-}\circ\cur{-\otimes k[G]}}(M)\\
        \cong&\cur{\textup{R}_{(H,K)}^{n}\Hom{H}{k,-}\circ\cur{-\otimes k[G]}}(M)\\
        \cong&\Ext{(H,K)}{n}\cur{k,M\otimes k[G]}.
    \end{align*}
\end{proof}
\subsection{}

We are now ready to give one more condition for when $\ind{H}{G}$ takes relative exact sequences to relative exact sequences.

\begin{proposition}\label{kgritise}
    Let $H$ and $K$ be closed subgroups of $G$. In the category of finite dimensional rational $G$-modules, $\textup{mod}\cur{G}$, the following are equivalent:
    \begin{enumerate}[(a)]
        \item $k[G]$ is $(H,H\cap K)$-injective.
        \item $\ind{H}{G}$ takes $(H,H\cap K)$-exact sequences to $(G,G)$-exact sequences.
        \item $\ind{H}{G}$ takes $(H,H\cap K)$-exact sequences to exact sequences.
    \end{enumerate}
\end{proposition}

\begin{proof}
    For $(a)\Rightarrow(b)$, consider an $(H,H\cap K)$-exact sequence
    \[
    \adjustbox{scale=1}{%
    \begin{tikzcd}
        0\arrow[r] & M_{1}\arrow[r] & M_{2}\arrow[r] & M_{3}\arrow[r] & 0. 
    \end{tikzcd}
    }
    \]
    $-\otimes k[G]$ is exact, so by \Cref{tetretre} $-\otimes k[G]$ sends $(H,K\cap K)$-exact sequences to $(H,K\cap K)$-exact sequences. By \Cref{triri}, $M_{i}\otimes k[G]$ is $(H,H\cap K)$-injective, so 
    \[
    \adjustbox{scale=1}{%
    \begin{tikzcd}
        0\arrow[r] & M_{1}\otimes k[G]\arrow[r] & M_{2}\otimes k[G]\arrow[r] & M_{3}\otimes k[G]\arrow[r] & 0 
    \end{tikzcd}
    }
    \]
    is split exact, and hence
    \[
    \adjustbox{scale=1}{%
    \begin{tikzcd}
        0\arrow[r] & \cur{M_{1}\otimes k[G]}^{H}\arrow[r] & \cur{M_{2}\otimes k[G]}^{H}\arrow[r] & \cur{M_{3}\otimes k[G]}^{H}\arrow[r] & 0 
    \end{tikzcd}
    }
    \]
    is also split exact.

    $(b)\Rightarrow(c)$ is obvious.

    For $(c)\Rightarrow(a)$, let $V$ be a finite dimensional $H$-module, by \Cref{mfdmie} and \Cref{retkgri} $$\Ext{(H,H\cap K)}{n}\cur{V,k[G]}\cong\Ext{(H,H\cap K)}{n}\cur{k,V^*\otimes k[G]}\cong\rRind{(H,H\cap K)}{n}{H}{G}\cur{V^*}=0$$ for $n>0$. Therefore, by \Cref{rie0}, $k[G]$ is $(H,H\cap K)$-injective in $\textup{mod}\cur{G}$.
\end{proof}

Note that in the category $\Mod{G}$, the proof of \Cref{kgritise} still works to show that $(a)\Rightarrow(b)\Rightarrow(c)$, but one does not necessarily have that $(c)\Rightarrow(a)$. One case where we do know that $(c)\Rightarrow(a)$, is if $\Mod{H\cap K}$ is semisimple.

\begin{proposition}
    Let $H$ and $K$ be closed subgroups of $G$ such that $\Mod{H\cap K}$ is semisimple. Then in the category $\textup{Mod}\cur{G}$ the following are equivalent:
    \begin{enumerate}[(a)]
        \item $k[G]$ is $(H,H\cap K)$-injective.
        \item $\ind{H}{G}$ takes $(H,H\cap K)$-exact sequences to $(G,G)$-exact sequences.
        \item $\ind{H}{G}$ takes $(H,H\cap K)$-exact sequences to exact sequences.
    \end{enumerate}
\end{proposition}

\begin{proof}
    From \Cref{kgritise} we already know that $(a)\Rightarrow(b)\Rightarrow(c)$. Assuming $(c)$, by \Cref{mfdmie} and \Cref{retkgri}, for all finite dimensional modules $V$, $$\Ext{\cur{H,H\cap K}}{n}\cur{V,k[G]}\cong\Ext{\cur{H,H\cap K}}{n}\cur{k,V^*\otimes k[G]}\cong\rRind{\cur{H,H\cap K}}{n}{H}{G}(V^*)=0\textup{ }\textup{ }\textup{ for }n>0.$$ Hence, from \Cref{kachssriwefd0}, $k[G]$ is $(H,H\cap K)$-injective.
\end{proof}
\subsection{}

We can now apply the above result to give us another condition for when the result of \Cref{rfr} holds.

\begin{proposition}\label{kgritiise}
    Let $H$ and $K$ be closed subgroups of $G$, $M$ a $G$-module, and $N$ an $H$-module. If $k[G]$ is $(H,H\cap K)$-injective, then $$\Ext{(G,K)}{i}\cur{M,\ind{H}{G}(N)}\cong\Ext{(H,H\cap K)}{i}\cur{M,N}\text{ for }i\geq0.$$
\end{proposition}

\begin{proof}
    First, by \Cref{iriri} $\ind{H}{G}$ takes $(H,H\cap K)$-injective modules to $(G,K)$-injective modules. From \Cref{kgritise}, $\ind{H}{G}$ takes $(H,H\cap K)$-exact sequences to $(G,K)$-exact sequences. So apply \Cref{crrdf2} to get
    \begin{align*}
        \Ext{(G,K)}{i}\cur{M,\ind{H}{G}(N)}&\cong\cur{\textup{R}_{(G,K)}^{i}\Hom{G}{M,-}\circ\ind{H}{G}(-)}(N)\\
        &\cong\textup{R}_{\cur{H,H\cap K}}^{i}\cur{\Hom{G}{M,-}\circ\ind{H}{G}(-)}(N)\\
        &\cong\textup{R}_{\cur{H,H\cap K}}^{i}\cur{\Hom{H}{M,-}}(N)\\
        &\cong\Ext{\cur{H,H\cap K}}{i}\cur{M,N}.
    \end{align*}
\end{proof}

Similarly, we also get another condition for when the result of \Cref{rrdisbsg} holds.

\begin{proposition}
    Let $H'$ be a closed subgroup of $G$, $H$ and $K$ be closed subgroups of $H'$, and $M$ an $H$-module. If $k[H']$ is $(H,H\cap K)$-injective, then $$\rRind{(H',K)}{i}{H'}{G}\cur{\ind{H}{H'}(M)}\cong\rRind{(H,H\cap K)}{i}{H}{G}\cur{M}\text{ for }i\geq0.$$
\end{proposition}

\begin{proof}
    First, by \Cref{iriri} $\ind{H}{H'}$ takes $(H,H\cap K)$-injective modules to $(H',K)$-injective modules. From \Cref{kgritise}, $\ind{H}{H'}$ takes $(H,H\cap K)$-exact sequences to $(H',K)$-exact sequences. So apply \Cref{crrdf2} to get
    \begin{align*}
        \rRind{(H',K)}{i}{H'}{G}\cur{\ind{H}{H'}(M)}&\cong\cur{\textup{R}_{(H',K)}^{i}\ind{H'}{G}(-)\circ\ind{H}{H'}(-)}(M)\\
        &\cong\textup{R}_{\cur{H,H\cap K}}^{i}\cur{\ind{H'}{G}(-)\circ\ind{H}{H'}(-)}(M)\\
        &\cong\textup{R}_{\cur{H,H\cap K}}^{i}\cur{\ind{H}{G}(-)}(M)\\
        &\cong\rRind{\cur{H,H\cap K}}{i}{H}{G}\cur{M}.
    \end{align*}
\end{proof}
\section{Factor Groups}\label{Factor Groups}

\subsection{}

Let $G$ be a algebraic group, $H$ a closed subgroup of $G$, and $N$ a normal subgroup of $G$. Then we have a natural map $\pi:G\to G/N$, and consequently its restriction, $\pi:H\to H/\cur{H\cap N}$. From these maps, there are also natural functors $\pi^*:\Mod{G/N}\to\Mod{G}$ and $\pi^*:\Mod{H/\cur{H\cap N}}\to\Mod{H}$.

Additionally, we also have natural functors $(-)^N:\Mod{G}\to\Mod{G/N}$ and $(-)^N=(-)^{H\cap N}:\Mod{H}\to\Mod{H/\cur{H\cap N}}$. One can see that $(-)^N\circ\pi^*=\id_{G/N}$, and $\pi^*\circ(-)^N=(-)^N$.
\subsection{}

Given the above functors we now want to examine their role when studying relative cohomology.

\begin{proposition}\label{fremre}
    Let $H$ be a closed subgroup of $G$, $N$ a normal subgroup of $G$, and $I$ be a $(G,H)$-injective module, then $I^{N}$ is $\cur{G/N,H/(H\cap N)}$-injective.
\end{proposition}

\begin{proof}
    Consider a $\cur{G/N,H/(H\cap N)}$-exact sequence
    \[
    \adjustbox{scale=1}{%
    \begin{tikzcd}
        0\arrow[r] & M_{1}\arrow[r, "d_{1}"] & M_{2}\arrow[r, "d_{2}"] & M_{3}\arrow[r] & 0 
    \end{tikzcd}
    }
    \]
    and a map $f:M_{1}\to I^{N}$. Apply $\pi^*$ to get
    \[
    \adjustbox{scale=1}{%
    \begin{tikzcd}
        0\arrow[r] & \pi^*M_{1}\arrow[r, "\pi^*d_{1}"]\arrow[d, "\pi^*f"] & \pi^*M_{2}\arrow[r, "\pi^*d_{2}"] & \pi^*M_{3}\arrow[r] & 0 \\
        & I^{N}
    \end{tikzcd}
    }.
    \]
    $\pi^*M_{i}$ and $\pi^*d_{i}$ are the same as $M_{i}$ and $d_{i}$ as vector spaces, so the sequence is still exact. Also, given an $H/\cur{H\cap N}$-homomorphism, $t$, one can see that $\pi^*t$ is an $H$-homomorphism which is the same as $t$ as a map of vector spaces. Hence the sequence is $(G,H)$-exact.

    Next, $I^{N}$ naturally sits inside $I$ as a submodule, so there is an injective $G$-homomorphism $i:I^{N}\to I$. Therefore, since $I$ is $(G,H)$-injective, there exists a map $h:M_{2}\to I$ such that the following diagram commutes
    \[
    \adjustbox{scale=1}{%
    \begin{tikzcd}
        0\arrow[r] & \pi^*M_{1}\arrow[r, "\pi^*d_{1}"]\arrow[d, "\pi^*f"] & \pi^*M_{2}\arrow[r, "\pi^*d_{2}"]\arrow[ldd, "h"] & \pi^*M_{3}\arrow[r] & 0 \\
        & I^{N}\arrow[d, hook, "i"] \\
        & I
    \end{tikzcd}
    }.
    \]
    Now apply $(-)^N$ to get
    \[
    \adjustbox{scale=1}{%
    \begin{tikzcd}
        0\arrow[r] & M_{1}\arrow[r, "d_{1}"]\arrow[d, "f"] & M_{2}\arrow[r, "d_{2}"]\arrow[ldd, "h^{N}"] & M_{3}\arrow[r] & 0 \\
        & I^{N}\arrow[d, hook, "i^{N}"] \\
        & I^{N}
    \end{tikzcd}
    },
    \]
    where $i^{N}$ is a isomorphism. Hence, $I^{N}$ is $\cur{G/N,H/(H\cap N)}$-injective.
\end{proof}

Now applying \Cref{fremre}, it can be shown that $\Mod{G/N}$ has enough relative injectives.

\begin{corollary}\label{gmnheri}
    Let $H$ be a closed subgroup of $G$ and $N$ a normal subgroup of $G$, then $\Mod{G/N}$ has enough $\cur{G/N,H/(H\cap N)}$-injectives.
\end{corollary}

\begin{proof}
    Let $M$ be a $G/N$-module, and so $\pi^*M$ is a $G$-module. By \Cref{riiffds} there exists a $(G,H)$-injective module $I$ such that $\pi^*M$ is a direct summand of $I$. So apply $(-)^N$ to get that $M$ is a direct summand of $I^N$, which is $\cur{G/N,H/(H\cap N)}$-injective by \Cref{fremre}.
\end{proof}
\subsection{}

Let $E$ be a finite dimensional $G$-module. We  move forward with our study by examining how the functor $\Hom{N}{E,-}:\Mod{G}\to\Mod{G/N}$ interacts with relative injective modules and relative exact sequences.

\begin{proposition}\label{hnritri}
    Let $H$ be a closed subgroup of $G$, $N$ a normal subgroup of $G$, $E$ a finite dimensional $G$-module, and $I$ a $(G,H)$-injective module. Then $\Hom{N}{E,I}$ is $\cur{G/N,H/(H\cap N)}$-injective.
\end{proposition}

\begin{proof}
    One has $\Hom{N}{E,I}\cong\cur{E^*\otimes I}^N$, so it suffices to show that $\cur{E^*\otimes I}^N$ is \\$\cur{G/N,H/(H\cap N)}$-injective, which follows by \Cref{triri} and \Cref{fremre}.
\end{proof}

\begin{proposition}\label{hnremn}
    Let $H$ be a closed subgroup of $G$, $N$ a normal subgroup of $G$, $E$ a finite dimensional $G$-module, and
    \[
    \adjustbox{scale=1}{%
    \begin{tikzcd}
        0\arrow[r] & M_{1}\arrow[r] & M_{2}\arrow[r] & M_{3}\arrow[r] & 0 
    \end{tikzcd}
    }
    \]
    be $(G,H)$-exact. If $N\leq H$, then
    \[
    \adjustbox{scale=1}{%
    \begin{tikzcd}
        0\arrow[r] & \Hom{N}{E,M_{1}}\arrow[r] & \Hom{N}{E,M_{2}}\arrow[r] & \Hom{N}{E,M_{3}}\arrow[r] & 0 
    \end{tikzcd}
    }
    \]
    is $(G/N,H/N)$-exact.
\end{proposition}

\begin{proof}
    Apply \Cref{hwsgoreie} and consider the second short exact sequence as $G/N$-modules.
\end{proof}
\subsection{}

We can now combine the above to get the following result. Note that \Cref{gmnheri} is important as it allows us to define $\Ext{\cur{G/N,H/N}}{n}$.

\begin{proposition}\label{refsre}
    Let $G$ be an algebraic group, $N$ a normal subgroup of $G$, $H$ a closed subgroup of $G$, and $M$, $E$, and $V$ be $G$-modules such that $E$ is finite dimensional. If $N\leq H$, then $$\Ext{\cur{G/N,H/N}}{n}\cur{M,\Hom{N}{E,V}}\cong\Ext{(G,H)}{n}\cur{M\otimes E,V}\text{ for }n\geq0.$$
\end{proposition}

\begin{proof}
    By \Cref{hnritri}, $\Hom{N}{E,-}$ takes $(G,H)$-injective modules to $\cur{G/N,H/N}$-\\injective modules, and by \Cref{hnremn} takes $(G,H)$-exact sequences to $\cur{G/N,H/N}$-exact sequences, so apply \Cref{crrdf2} to get
    \begin{align*}
        \Ext{\cur{G/N,H/N}}{n}\cur{M,\Hom{N}{E,V}}&\cong\cur{\textup{R}_{\cur{G/N,H/N}}^{n}\Hom{G/N}{M,-}\circ\Hom{N}{E,-}}(V)\\
        &\cong\textup{R}_{\cur{G,H}}^{n}\cur{\Hom{G/N}{M,-}\circ\Hom{N}{E,-}}(V)\\
        &\cong\textup{R}_{\cur{G,H}}^{n}\Hom{G}{M\otimes E,-}(V)\\
        &\cong\Ext{(G,H)}{n}\cur{M\otimes E,V}.
    \end{align*}
\end{proof}

Note that taking $E=k$ in \Cref{refsre} yields the following:

\begin{corollary}
    Let $G$ be an algebraic group, $N$ a normal subgroup of $G$, $H$ a closed subgroup of $G$, and $M$ and $V$ be $G$-modules. If $N\leq H$, then for $n\geq0$, $$\Ext{G/N}{n}\cur{k,V^{N}}\cong\Ext{\cur{G/N,\{1\}}}{n}\cur{k,V^{N}}\cong\Ext{\cur{G/N,\{1\}}}{n}\cur{k,\Hom{N}{k,V}}\cong\Ext{(G,N)}{n}\cur{k,V}.$$
\end{corollary}

We remark that in \cite[Theorem 2.6]{Kim65}, the above is proven when $\textup{char}\cur{k}=0$.

\begin{proposition}
    Let $G$ be an algebraic group, $N$ a normal diagonalisable subgroup of $G$ (so $\Mod{N}$ is semi-simple), and $M$, $E$, and $V$ be $G$-modules such that $E$ is finite dimensional. Then $$\Ext{(G,N)}{n}\cur{M\otimes E,V}\cong\Ext{G}{n}\cur{M\otimes E,V}\text{ for }n\geq0,$$ and in particular, $$\Ext{(G,N)}{n}\cur{M,V}\cong\Ext{G}{n}\cur{M,V}\text{ for }n\geq0.$$
\end{proposition}

\begin{proof}
    By \Cref{refsre}, $\Ext{\cur{G/N,\{1\}}}{n}\cur{M,\Hom{N}{E,V}}\cong\Ext{(G,N)}{n}\cur{M\otimes E,V}$, and from the Lyndon-Hochschild-Serre spectral sequence (in particular \cite[I.6.8]{Jan03}) $$\Ext{\cur{G/N,\{1\}}}{n}\cur{M,\Hom{N}{E,V}}\cong\Ext{G}{n}\cur{M\otimes E,V}.$$ Hence, $$\Ext{(G,N)}{n}\cur{M\otimes E,V}\cong\Ext{G}{n}\cur{M\otimes E,V}.$$ Now taking $E=k$, we get $\Ext{(G,N)}{n}\cur{M,V}\cong\Ext{G}{n}\cur{M,V}$.
\end{proof}
\subsection{}

When considering the normal case of right derived functors of $\ind{H}{G}$, it is known that $$\rRind{}{n}{H'/N}{G/N}\cur{M}\cong\rRind{}{n}{H'}{G}\cur{M},$$ see \cite[II.6.11]{Jan03}. We now prove that there is also a relative version of this.

\begin{lemma}
    Let $N$, $H$, and $H'$ be closed subgroups of $G$ such that $N$ is normal and $N\leq H\leq H'\leq G$. Then for any $H'/N$-module $M$, $$\rRind{(H'/N,H/N)}{n}{H'/N}{G/N}\cur{M}\cong\rRind{(H',H)}{n}{H'}{G}\cur{M}\text{ for }n\geq0.$$
\end{lemma}

\begin{proof}
    Note that this is the same as showing that $$\cur{\pi^*\circ\rRind{(H'/N,H/N)}{n}{H'/N}{G/N}}\cur{M}\cong\cur{\rRind{\cur{H',H}}{n}{H'}{G}\circ\pi^*}\cur{M}.$$
    
    From \cite[I.6.11(2')]{Jan03}, $\pi^*\circ\ind{H'/N}{G/N}\cong\ind{H'}{G}\circ\pi^*$, so we need to show that we can apply both \Cref{crrdf1} and \Cref{crrdf2}.

    First, it is clear that $\pi^*$ is exact, so one can apply \Cref{crrdf1} to $\pi^*\circ\ind{H'/N}{G/N}$. Next, we need to show that $\pi^*$ takes $\cur{H'/N,H/N}$-injective modules to $\cur{H',H}$-injective modules. Let $M$ be $\cur{H'/N,H/N}$-injective, then by \Cref{riiffds}, $M$ is a direct summand of some $\ind{H/N}{H'/N}\cur{I}$. Hence, $\pi^*M$ is a direct summand of $\pi^*\ind{H/N}{H'/N}\cur{I}\cong\ind{H}{H'}\cur{\pi^*I}$, which is $\cur{H',H}$-injective by \Cref{iri}. Therefore, by \Cref{dsim}, $\pi^*M$ is $\cur{H',H}$-injective.

    Last, we must show that $\pi^*$ takes $\cur{H'/N,H/N}$-exact sequences to $\cur{H',H}$-exact sequences. So consider a $\cur{H'/N,H/N}$-exact sequence
    \[
    \adjustbox{scale=1}{%
    \begin{tikzcd}
        0\arrow[r] & M_{1}\arrow[r, "d_{1}"] & M_{2}\arrow[r, "d_{2}"] & M_{3}\arrow[r] & 0.
    \end{tikzcd}
    }
    \]
    Then there exists $H/N$-homomorphisms $h_{i}:M_{i}\to M_{i-1}$ such that $\id_{M_{1}}=h_{2}\circ d_{1}$, $\id_{M_{2}}=d_{1}\circ h_{2}+h_{3}\circ d_{2}$, and $\id_{M_{3}}=d_{2}\circ h_{3}$. Now apply $\pi^*$ to everything to get an exact sequence of $H'$-modules
    \[
    \adjustbox{scale=1}{%
    \begin{tikzcd}
        0\arrow[r] & \pi^*M_{1}\arrow[r, "\pi^*d_{1}"] & \pi^*M_{2}\arrow[r, "\pi^*d_{2}"] & \pi^*M_{3}\arrow[r] & 0.
    \end{tikzcd}
    }
    \]
    Consider the $H$-homomorphisms $\pi^*h_{i}:\pi^*M_{i}\to\pi^*M_{i-1}$, and since $\pi^*$ does not change the structure as a vector space, it is clear that $\id_{M_{1}}=\pi^*h_{2}\circ \pi^*d_{1}$, $\id_{M_{2}}=\pi^*d_{1}\circ \pi^*h_{2}+\pi^*h_{3}\circ \pi^*d_{2}$, and $\id_{M_{3}}=\pi^*d_{2}\circ \pi^*h_{3}$. Hence, the sequence is $\cur{H',H}$-exact.

    Therefore, apply \Cref{crrdf1} and \Cref{crrdf2} to get
    \begin{align*}
        \cur{\pi^*\circ\rRind{(H'/N,H/N)}{n}{H'/N}{G/N}}\cur{M}&\cong\textup{R}_{(H'/N,H/N)}^{n}\cur{\pi^*\circ\ind{H'/N}{G/N}}\cur{M}\\
        &\cong\textup{R}_{(H'/N,H/N)}^{n}\cur{\ind{H'}{G}\circ\pi^*}\cur{M}\\
        &\cong\cur{\rRind{\cur{H',H}}{n}{H'}{G}\circ\pi^*}\cur{M}.
    \end{align*}
\end{proof}
\subsection{}

We will now consider the functors $\mc{F}_1$ and $\mc{F}_2$ from $\Mod{H}$ to $\Mod{G/N}$ defined by $$\mc{F}_1(M)=\cur{\ind{H}{G}(M)}^N\hspace{1cm}\textup{ and }\hspace{1cm}\mc{F}_2(M)=\ind{H/(H\cap N)}{G/N}\cur{M^{H\cap N}}.$$

First, we focus on $\mc{F}_1(M)$, and give a equivalent formulation for $\textup{R}_{(H,H\cap K)}^{n}\mc{F}_1(M)$ when under the right assumptions.

\begin{proposition}\label{HKNGHrHci}
    Let $H$, $K$, and $N$ be closed subgroups of $G$ such that $N$ is normal and such that one of the following holds:
    \begin{enumerate}[(a)]
        \item $HK=G$,
        \item $k[G]$ is $(H,H\cap K)$-injective, or
        \item $\Mod{H}$ is semisimple.
    \end{enumerate}
    Then $$\textup{R}_{(G,K)}^{n}\Hom{N}{k,\ind{H}{G}(M)}\cong\textup{R}_{(H,H\cap K)}^{n}\mc{F}_1(M)\text{ for }n\geq0.$$
\end{proposition}

\begin{proof}
    By \Cref{iriri}, $\ind{H}{G}$ takes $\cur{H,H\cap K}$-injective modules to $\cur{G,K}$-injective modules.
    
    If $HK=G$ or $k[G]$ is $(H,H\cap K)$-injective, by \Cref{iores} and \Cref{kgritise}, $\ind{H}{G}$ takes $\cur{H,H\cap K}$-exact sequences to $\cur{G,K}$-exact sequences.
    
    If $\Mod{H}$ is semisimple, apply the same reasoning as in \Cref{reisshkss} to see that $\ind{H}{G}$ takes $\cur{H,H\cap K}$-exact sequences to $\cur{G,K}$-exact sequences.

    Therefore, in any case, apply \Cref{crrdf2} to get
    \begin{align*}
        \textup{R}_{(G,K)}^{n}\Hom{N}{k,\ind{H}{G}(M)}&\cong\cur{\textup{R}_{(G,K)}^{n}\Hom{N}{k,-}\circ\ind{H}{G}(-)}(M)\\
        &\cong\textup{R}_{(H,H\cap K)}^{n}\cur{\Hom{N}{k,-}\circ\ind{H}{G}(-)}(M)\\
        &\cong\textup{R}_{(H,H\cap K)}^{n}\cur{(-)^N\circ\ind{H}{G}(-)}(M)\\
        &\cong\textup{R}_{(H,H\cap K)}^{n}\mc{F}_1(M).
    \end{align*}
\end{proof}

Next, we look at $\mc{F}_2(M)$, and give a equivalent formulation for $\textup{R}_{(H,K)}^{n}\mc{F}_2(M)$ when under the right assumptions.

\begin{proposition}\label{KHNGricH}
    Let $H$, $K$, and $N$ be closed subgroups of $G$ such that $N$ is normal, $K\leq H$, and $H\cap N\leq K$. Then $$\rRind{\cur{H/(H\cap N),K/(H\cap N)}}{n}{H/(H\cap N)}{G/N}\cur{\Hom{H\cap N}{k,M}}\cong\textup{R}_{(H,K)}^{n}\mc{F}_2(M)\text{ for }n\geq0.$$
\end{proposition}

\begin{proof}
    By \Cref{hnritri} and \Cref{hnremn}, $\Hom{H\cap N}{k,-}$ takes $(H,K)$-injective modules to $\cur{H/(H\cap K),K/(H\cap K)}$-injective modules and $(H,K)$-exact sequences to \\$\cur{H/(H\cap K),K/(H\cap K)}$-exact sequences. Therefore, apply \Cref{crrdf2} to get
    \begin{align*}
        &\rRind{\cur{H/(H\cap N),K/(H\cap N)}}{n}{H/(H\cap N)}{G/N}\cur{\Hom{H\cap N}{k,M}}\\
        \cong&\cur{\rRind{\cur{H/(H\cap N),K/(H\cap N)}}{n}{H/(H\cap N)}{G/N}(-)\circ\Hom{H\cap N}{k,-}}(M)\\
        \cong&\textup{R}_{(H,K)}^{n}\cur{\ind{H/(H\cap N)}{G/N}(-)\circ\Hom{H\cap N}{k,-}}(M)\\
        \cong&\textup{R}_{(H,K)}^{n}\cur{\ind{H/(H\cap N)}{G/N}(-)\circ(-)^{H\cap N}}(M)\\
        \cong&\textup{R}_{(H,K)}^{n}\mc{F}_2(M).
    \end{align*}
\end{proof}

Now applying \Cref{HKNGHrHci} and \Cref{KHNGricH} will yield the following.

\begin{proposition}
    Let $H$, $K$, and $N$ be closed subgroups of $G$ such that $N$ is normal, $H\cap N\leq H\cap K$ and such that one of the following holds:
    \begin{enumerate}[(a)]
        \item $HK=G$,
        \item $k[G]$ is $(H,H\cap K)$-injective, or
        \item $\Mod{H}$ is semisimple.
    \end{enumerate}
    Then for $n\geq0$, $$\textup{R}_{(G,K)}^{n}\Hom{N}{k,\ind{H}{G}(M)}\cong\rRind{\cur{H/(H\cap N),(H\cap K)/(H\cap N)}}{n}{H/(H\cap N)}{G/N}\cur{\Hom{H\cap N}{k,M}}.$$
\end{proposition}

\begin{proof}
    From \cite[Proposition I.6.12(a)]{Jan03}, $\mc{F}_1\cong\mc{F}_2$, so clearly $$\textup{R}_{(H,H\cap K)}^{n}\mc{F}_1(M)\cong\textup{R}_{(H,H\cap K)}^{n}\mc{F}_2(M).$$ Then apply \Cref{HKNGHrHci} and \Cref{KHNGricH} to get the result.
\end{proof}
\section{Applications and Examples}\label{Examples}

\subsection{}

We start off by examining what happens when $\textup{char}(k)=0$. Let $G$ be a reductive algebraic group over an algebraically closed field $k$ of characteristic $0$, $H$ a closed subgroup of $G$, and $M$ and $N$ be finite dimensional $G$-modules. It is well known that $\Ext{G}{i}\cur{M,N}=0$ for $i>0$. Hence, $N$ is an injective module in $\textup{mod}(G)$, the category of finite dimensional rational $G$-modules. Now by \Cref{rimu}, $N$ will also be a $(G,H)$-injective module in $\textup{mod}(G)$. Therefore, it is clear from \Cref{rie0}, that $$\Ext{(G,H)}{i}\cur{M,N}=0\textup{ for }i>0$$ for any finite dimensional $G$-module $M$.

\subsection{}

We now go back to arbitrary characteristic and start by looking at relative cohomology for a parabolic subgroup of a reductive algebraic group.

Let $G$ be a reductive algebraic group, $B$ a Borel subgroup, $P\geq B$ a parabolic subgroup, and $L$ a Levi factor of $P$, then from \Cref{rfr} $$\Ext{(P,L)}{i}\cur{M,\ind{B}{P}(N)}\cong\Ext{(B,B\cap L)}{i}\cur{M,N}\textup{ for }i\geq0.$$

\subsection{}

Let $G$ be a reductive algebraic group, $B$ a Borel subgroup, $P\geq B$ a parabolic subgroup, $L$ a Levi factor of $P$, and $U$ the unipotent subgroup of $P$ such that $P=L\ltimes U$, then from \Cref{rfr} $$\Ext{(P,L)}{i}\cur{M,\ind{U}{P}(N)}\cong\Ext{(U,\{1\})}{i}\cur{M,N}\textup{ for }i\geq0.$$

\subsection{}

Let $G$ be a reductive algebraic group, $B$ a Borel subgroup, $P\geq B$ a parabolic subgroup, $L$ a Levi factor of $P$, and $U$ the unipotent subgroup of $P$ such that $P=L\ltimes U$, then from \Cref{rfr} $$\Ext{(P,U)}{i}\cur{M,\ind{L}{P}(N)}\cong\Ext{(L,\{1\})}{i}\cur{M,N}\textup{ for }i\geq0.$$

Now take $P=B$, and so $L=T$ is a maximal torus. Hence we obtain that $$\Ext{(B,U)}{i}\cur{M,\ind{T}{B}(N)}\cong\Ext{(T,\{1\})}{i}\cur{M,N}\textup{ for }i\geq0,$$ which is equal to $0$ when $i>0$.

\subsection{}

We now move to looking at relative cohomology for a reductive algebraic group $G$. Let $H$ be a closed subgroup of $G$, $X_+$ be the set of dominant weights for $G$, and $L(\lambda)$ be the simple $G$-module of highest weight $\lambda\in X_+$. Then for all $\lambda\in X_{+}$, $$\Ext{(G,H)}{1}\cur{L\cur{\lambda},L\cur{\lambda}}=0.$$

From \cite[Theorem 2.2]{Kim66}, $\Ext{(G,H)}{1}\cur{L(\lambda),L(\lambda)}$ is isomorphic to equivalence classes of extensions
\[
\adjustbox{scale=1}{%
\begin{tikzcd}
    0\arrow[r] & L(\lambda)\arrow[r] & A\arrow[r] & L(\lambda)\arrow[r] & 0
\end{tikzcd}
}
\]
that are $(G,H)$-exact. But from \cite[II.2.12(1)]{Jan03}, any exact sequence of the form
\[
\adjustbox{scale=1}{%
\begin{tikzcd}
    0\arrow[r] & L(\lambda)\arrow[r] & A\arrow[r] & L(\lambda)\arrow[r] & 0
\end{tikzcd}
}
\]
must be split. Hence, $\Ext{(G,H)}{1}\cur{L\cur{\lambda},L\cur{\lambda}}=0$.

\subsection{}

Again, from \cite[Theorem 2.2]{Kim66}, $\Ext{(G,H)}{1}\cur{M,N}$ is isomorphic to equivalence classes of extensions of $N$ by $M$ that are $(G,H)$-exact. Now consider the anti-automorphism $\tau$. Then for any $G$-module $M$, we can define $^\tau M$. It is clear that $\res{H}{G}\cur{^\tau M}\cong\textup{}^\tau \res{H}{G}\cur{M}$. Hence, given a $(G,H)$-exact sequence
\[
\adjustbox{scale=1}{%
\begin{tikzcd}
    0\arrow[r] & M_{1}\arrow[r, "f"] & M_{2}\arrow[r, "g"] & M_{3}\arrow[r] & 0,
\end{tikzcd}
}
\]
it is clear that
\[
\adjustbox{scale=1}{%
\begin{tikzcd}
    0\arrow[r] & ^\tau M_{3}\arrow[r, "^\tau g"] & ^\tau M_{2}\arrow[r, "^\tau f"] & ^\tau M_{1}\arrow[r] & 0
\end{tikzcd}
}
\]
is also $(G,H)$-exact. So there is an isomorphism $$\Ext{(G,H)}{1}\cur{M_{3},M_{1}}\cong\Ext{(G,H)}{1}\cur{\textup{}^\tau M_{1},\textup{}^\tau M_{3}}$$ and in particular, $\Ext{(G,H)}{1}\cur{L\cur{\lambda},L\cur{\mu}}\cong\Ext{(G,H)}{1}\cur{L\cur{\mu},L\cur{\lambda}}$ for all $\lambda,\mu\in X_{+}$.

\subsection{}

By \Cref{rfr}, $\Ext{(G,G)}{i}\cur{M,\ind{H}{G}(N)}\cong\Ext{(H,H)}{i}\cur{M,N}$ for $i\geq0$. From the definition of relative right derived functors it is clear that $$\Ext{(G,G)}{i}\cur{M,\ind{H}{G}(N)}=
\begin{cases}
    \Hom{G}{M,\ind{H}{G}\cur{N}} & \textup{if }i=0\\
    0 & \textup{otherwise}
\end{cases}$$ and $$\Ext{(H,H)}{i}\cur{M,N}=
\begin{cases}
    \Hom{H}{N,M} & \textup{if }i=0\\
    0 & \textup{otherwise}
\end{cases}.$$ Hence, $\Hom{G}{M,\ind{H}{G}\cur{N}}\cong\Hom{H}{M,N}$, which is just Frobenius Reciprocity.

\subsection{}

The following example is similar to the Ext-transfer property, $$\Ext{G}{i}\cur{M,N}\cong\Ext{P}{i}\cur{M,N}\cong\Ext{B}{i}\cur{M,N}\textup{ for }i\geq0,$$ which can be found in \cite[II.4.7]{Jan03}.

\begin{theorem}
    Let $G$ be an algebraic group, $M$ and $N$ be $G$-modules, $P$ be a parabolic subgroup of $G$, and $K$ a closed subgroup of $G$ such that one of the following holds
    \begin{enumerate}[(a)]
        \item $PK=G$ and $\cur{P\cap K}B=P$,
        \item $\Mod{K}$ is semisimple, or
        \item $k[G]$ is $\cur{P,P\cap K}$-injective and $\cur{B,B\cap K}$-injective.
    \end{enumerate}
    Then
    \begin{equation}
        \Ext{(G,K)}{i}\cur{M,N}\cong\Ext{(P,P\cap K)}{i}\cur{M,N}\cong\Ext{(B,B\cap K)}{i}\cur{M,N}\textup{ for }i\geq0.
    \end{equation}
\end{theorem}

\begin{proof}
    If $\Mod{K}$ is semisimple, by \Cref{rtt}
    \[
    \adjustbox{scale=1}{%
    \begin{tikzcd}[
        column sep=tiny,row sep=small,
        ar symbol/.style = {draw=none,"#1" description,sloped},
        isomorphic/.style = {ar symbol={\cong}},
        equals/.style = {ar symbol={=}},
        ]
        \Ext{G}{i}\cur{M,N}\arrow[r,isomorphic]\arrow[d,isomorphic] & \Ext{P}{i}\cur{M,N}\arrow[r,isomorphic]\arrow[d,isomorphic] & \Ext{B}{i}\cur{M,N}\arrow[d,isomorphic] \\
        \Ext{(G,K)}{i}\cur{M,N} & \Ext{(P,P\cap K)}{i}\cur{M,N} & \Ext{(B,B\cap K)}{i}\cur{M,N}
    \end{tikzcd}
    }.
    \]

    Otherwise, by \Cref{rfr} and \Cref{kgritiise}, $$\Ext{(G,K)}{i}\cur{M,\ind{P}{G}\cur{N}}\cong\Ext{(P,P\cap K)}{i}\cur{M,N}\textup{ for }i\geq0.$$

    We can also use the same argument replacing $G$ (resp. $P$) with $P$ (resp. $B$), where $B\leq P$ is a Borel subgroup.
\end{proof}
\subsection{CPS Coupled Parabolic Systems}

For a finite group $G$ and subgroups $H$ and $K$ of $G$, the Mackey decomposition theorem gives a nice description of $\res{H}{G}\cur{\ind{K}{G}\cur{M}}$ for a $K$-module $M$. Such a result does not hold in general for an arbitrary algebraic group $G$. In \cite{CPS83}, a Mackey imprimitivity theory for arbitrary algebraic groups is developed. One consequence, \cite[Example 4.5]{CPS83}, is a description of $\res{P_{J_{2}}}{G}\cur{\ind{P_{J_{1}}}{G}\cur{M}}$, where the $P_{J_{i}}$ are ``nice'' parabolic subgroups of $G$ (we will explain what ``nice'' means later in this section). Such a pair of parabolic subgroups gives what is called a {\it CPS Coupled Parabolic System}. We now want to incorporate these CPS coupled parabolic systems into our study of relative cohomology.

Let $G$ be an algebraic group, $H$ a closed subgroup, $M$ an $H$-module, and $x\in G$. Define $H^x=x^{-1}Hx$ and $M^x$ as a vector space to be isomorphic to $M$ with action from $H^x$ defined as $h\cdot m=\cur{xhx^{-1}}m$ for $h\in H^x$.

\begin{lemma}\label{iotts}
    Let $G$  be an algebraic group, $H$ and $K$ closed subgroups such that $K\subseteq H$ and $K^{x}=x^{-1}Kx\subseteq H$ for $x\in G$, then $\ind{K}{H}\cur{M}\cong\ind{K^{x}}{H}\cur{M^{x}}$.
\end{lemma}

\begin{proof}
    First recall that
    \begin{equation*}
        \ind{K}{H}\cur{M}=\crly{f:H\to M:\stackanchor{f(hy)=y^{-1}f(h)}{\textup{for all }h\in H\textup{ and all }y\in K}}
    \end{equation*}
    and
    \begin{equation*}
        \ind{K^{x}}{H}\cur{M^{x}}=\crly{g:H\to M^{x}:\stackanchor{g(hx^{-1}yx)=x^{-1}y^{-1}x\cdot g(h)=y^{-1}g(h)}{\textup{for all }h\in H\textup{ and all }y\in K}}.
    \end{equation*}

    We define a map $\phi:\ind{K}{H}\cur{M}\to\ind{K^{x}}{H}\cur{M^{x}}$ where $\phi\cur{f}\cur{h}=f\cur{hx^{-1}}$. Then
    \begin{align*}
        \phi\cur{f}\cur{hx^{-1}yx}=f\cur{hx^{-1}yxx^{-1}}=f\cur{hx^{-1}y}=y^{-1}f\cur{hx^{-1}}=x^{-1}y^{-1}x\cdot f\cur{hx^{-1}}\\
        =x^{-1}y^{-1}x\cdot\phi\cur{f}\cur{h}
    \end{align*}
    for all $h\in H$ and all $y\in K$, and so $\phi\cur{f}\in\ind{K^{x}}{H}\cur{M^{x}}$. Next one can see that $$\phi\cur{h'*f}\cur{h}=\cur{h'*f}\cur{hx^{-1}}=f\cur{h'hx^{-1}}=\phi\cur{f}\cur{h'h}=\cur{h'*\phi\cur{f}}\cur{h}$$ and $$\phi\cur{f_1+f_2}\cur{h}=\cur{f_1+f_2}\cur{hx^{-1}}=f_1\cur{hx^{-1}}+f_2\cur{hx^{-1}}=\phi\cur{f_1}\cur{h}+\phi\cur{f_2}\cur{h}.$$ Hence, $\phi$ is an $H$-homomorphism.

    Now we define a map $\psi:\ind{K^{x}}{H}\cur{M^{x}}\to\ind{K}{H}\cur{M}$ where $\psi\cur{g}\cur{h}=g\cur{hx}$. Therefore, $$\psi\cur{g}\cur{hy}=g\cur{hyx}=g\cur{hxx^{-1}yx}=x^{-1}y^{-1}x\cdot g\cur{hx}=y^{-1}g\cur{hx}=y^{-1}\psi\cur{g}\cur{h}$$ for all $h\in H$ and all $y\in K$, and so $\psi\cur{g}\in\ind{K}{H}\cur{M}$. Next one can see that $$\psi\cur{h'*g}\cur{h}=\cur{h'*g}\cur{hx}=g\cur{h'hx}=\psi\cur{g}\cur{h'h}=\cur{h'*\psi\cur{g}}\cur{h}$$ and $$\psi\cur{g_1+g_2}\cur{h}=\cur{g_1+g_2}\cur{hx}=g_1\cur{hx}+g_2\cur{hx}=\psi\cur{g_1}\cur{h}+\psi\cur{g_2}\cur{h}.$$ Hence, $\psi$ is an $H$-homomorphism.

    Last, it is easy to see that for all $h\in H$ $$\psi\cur{\phi\cur{f}}\cur{h}=\phi\cur{f}\cur{hx}=f\cur{hxx^{-1}}=f\cur{h}$$ and $$\phi\cur{\psi\cur{g}}\cur{h}=\psi\cur{g}\cur{hx^{-1}}=g\cur{hx^{-1}x}=g\cur{h}.$$ Therefore, $\ind{K}{H}\cur{M}\cong\ind{K^{x}}{H}\cur{M^{x}}$.
\end{proof}

We will use the following assumptions, which come from \cite[4.5]{CPS83}, for the remainder of this section. Let $G$ be a connected semisimple algebraic group, $B$ a Borel subgroup, $S$ a corresponding set of simple roots, $J_{1}$ and $J_{2}$ proper subsets of $S$ such that $J_{1}^*\cup J_{2}=S$, where $J_{1}^*=-w_0\cur{J_{1}}$ and $w_{0}$ is the long word in the Weyl group $W=N_G(T)/T$.

Now by applying \Cref{iotts} and \cite[Example 4.5]{CPS83}, we obtain the following result which describes how $\ind{P_{J_{1}}}{G}$ interacts with $\cur{P_{J_{1}},P_{J_{1}}\cap P_{J_{2}}^{w_0}}$-exact sequences.

\begin{lemma}\label{iotres}
    If
    \[
    \adjustbox{scale=1}{%
    \begin{tikzcd}
        \cdots\arrow[r] & M_{i-1}\arrow[r, "d_{i-1}"] & M_{i}\arrow[r, "d_{i}"] & M_{i+1}\arrow[r, "d_{i+1}"] & \cdots
    \end{tikzcd}
    }
    \]
    is a $\cur{P_{J_{1}},P_{J_{1}}\cap P_{J_{2}}^{w_0}}$-exact sequence, then
    \[
    \adjustbox{scale=1}{%
    \begin{tikzcd}[row sep=huge]
        \cdots\arrow[r] & \ind{P_{J_{1}}}{G}\cur{M_{i-1}}\arrow[rr, "\ind{P_{J_{1}}}{G}\cur{d_{i-1}}"]\arrow[d, phantom, ""{coordinate, name=Z0}] &  & \ind{P_{J_{1}}}{G}\cur{M_{i}}\arrow[dll, swap, "\ind{P_{J_{1}}}{G}\cur{d_{i}}", rounded corners, to path={ -- ([xshift=2ex]\tikztostart.east)|- (Z0) [near end]\tikztonodes-|([xshift=-2ex]\tikztotarget.west) -- (\tikztotarget)}] &  & \\
        & \ind{P_{J_{1}}}{G}\cur{M_{i+1}}\arrow[rr, "\ind{P_{J_{1}}}{G}\cur{d_{i+1}}"] &  & \cdots
    \end{tikzcd}
    }
    \]
    is $\cur{G,P_{J_2}^{w_0}}$-exact.
\end{lemma}

\begin{proof}
    Consider a $\cur{P_{J_{1}},P_{J_{1}}\cap P_{J_{2}}^{w_0}}$-exact sequence
    \[
    \adjustbox{scale=1}{%
    \begin{tikzcd}
        \cdots\arrow[r] & M_{i-1}\arrow[r, "d_{i-1}"] & M_{i}\arrow[r, "d_{i}"] & M_{i+1}\arrow[r, "d_{i+1}"] & \cdots.
    \end{tikzcd}
    }
    \]
    By \Cref{iotts}, the sequence

    \[
    \adjustbox{scale=1}{%
    \begin{tikzcd}[row sep=huge]
        \cdots\arrow[r] & \ind{P_{J_{1}}}{G}\cur{M_{i-1}}\arrow[rr, "\ind{P_{J_{1}}}{G}\cur{d_{i-1}}"]\arrow[d, phantom, ""{coordinate, name=Z0}] &  & \ind{P_{J_{1}}}{G}\cur{M_{i}}\arrow[dll, swap, "\ind{P_{J_{1}}}{G}\cur{d_{i}}", rounded corners, to path={ -- ([xshift=2ex]\tikztostart.east)|- (Z0) [near end]\tikztonodes-|([xshift=-2ex]\tikztotarget.west) -- (\tikztotarget)}] &  & \\
        & \ind{P_{J_{1}}}{G}\cur{M_{i+1}}\arrow[rr, "\ind{P_{J_{1}}}{G}\cur{d_{i+1}}"] &  & \cdots
    \end{tikzcd}
    }
    \]
    is isomorphic to
    \[
    \adjustbox{scale=1}{%
    \begin{tikzcd}[row sep=huge]
        \cdots\arrow[r] & \ind{P_{J_{1}}^{w_0}}{G}\cur{\cur{M_{i-1}}^{w_0}}\arrow[rrr, "\ind{P_{J_{1}}^{w_0}}{G}\cur{\cur{d_{i-1}}^{w_0}}"]\arrow[d, phantom, ""{coordinate, name=Z0}] &  &  & \ind{P_{J_{1}}^{w_0}}{G}\cur{\cur{M_{i}}^{w_0}}\arrow[dlll, swap, "\ind{P_{J_{1}}^{w_0}}{G}\cur{\cur{d_{i}}^{w_0}}", rounded corners, to path={ -- ([xshift=2ex]\tikztostart.east)|- (Z0) [near end]\tikztonodes-|([xshift=-2ex]\tikztotarget.west) -- (\tikztotarget)}] \\
        & \ind{P_{J_{1}}^{w_0}}{G}\cur{\cur{M_{i+1}}^{w_0}}\arrow[rrr, "\ind{P_{J_{1}}^{w_0}}{G}\cur{\cur{d_{i+1}}^{w_0}}"] &  &  & \cdots
    \end{tikzcd}
    }.
    \]

    Now, from \cite[Example 4.5]{CPS83}, $$\res{P_{J_{2}}^{w_0}}{G}\cur{\ind{P_{J_{1}}^{w_0}}{G}\cur{M^{w_0}}}\cong\ind{P_{J_{1}}\cap P_{J_{2}}^{w_0}}{P_{J_{2}}^{w_0}}\cur{\res{P_{J_{1}}\cap P_{J_{2}}^{w_0}}{P_{J_{1}}}\cur{M}},$$ hence the following sequences are isomorphic,
    \[
    \adjustbox{scale=1}{%
    \begin{tikzcd}
        \cdots\arrow[r] & 
        \ind{P_{J_{1}}\cap P_{J_{2}}^{w_0}}{P_{J_{2}}^{w_0}}\cur{\res{P_{J_{1}}\cap P_{J_{2}}^{w_0}}{P_{J_{1}}}\cur{M_{i-1}}}\arrow[r]\arrow[d, phantom, ""{coordinate, name=Z0}] & \ind{P_{J_{1}}\cap P_{J_{2}}^{w_0}}{P_{J_{2}}^{w_0}}\cur{\res{P_{J_{1}}\cap P_{J_{2}}^{w_0}}{P_{J_{1}}}\cur{M_{i}}}\arrow[dl, rounded corners, to path={ -- ([xshift=2ex]\tikztostart.east)|- (Z0) [near end]\tikztonodes-|([xshift=-2ex]\tikztotarget.west) -- (\tikztotarget)}] & \\
        & \ind{P_{J_{1}}\cap P_{J_{2}}^{w_0}}{P_{J_{2}}^{w_0}}\cur{\res{P_{J_{1}}\cap P_{J_{2}}^{w_0}}{P_{J_{1}}}\cur{M_{i+1}}}\arrow[r] & \cdots
    \end{tikzcd}
    }
    \]
    and
    \[
    \adjustbox{scale=1}{%
    \begin{tikzcd}
        \cdots\arrow[r] & \res{P_{J_{2}}^{w_0}}{G}\cur{\ind{P_{J_{1}}^{w_0}}{G}\cur{\cur{\cur{M_{i-1}}}^{w_0}}}\arrow[r]\arrow[d, phantom, ""{coordinate, name=Z0}] & \res{P_{J_{2}}^{w_0}}{G}\cur{\ind{P_{J_{1}}^{w_0}}{G}\cur{\cur{M_{i}}^{w_0}}}\arrow[dl, swap, "\ind{P_{J_{1}}^{w_0}}{G}\cur{\cur{d_{i}}^{w_0}}", rounded corners, to path={ -- ([xshift=2ex]\tikztostart.east)|- (Z0) [near end]\tikztonodes-|([xshift=-2ex]\tikztotarget.west) -- (\tikztotarget)}] \\
        & \res{P_{J_{2}}^{w_0}}{G}\cur{\ind{P_{J_{1}}^{w_0}}{G}\cur{\cur{M_{i+1}}^{w_0}}}\arrow[r] & \cdots
    \end{tikzcd}
    }.
    \]
    From our assumptions, the former is split. Therefore,
    \[
    \adjustbox{scale=1}{%
    \begin{tikzcd}[row sep=huge]
        \cdots\arrow[r] & \ind{P_{J_{1}}^{w_0}}{G}\cur{\cur{M_{i-1}}^{w_0}}\arrow[rrr, "\ind{P_{J_{1}}^{w_0}}{G}\cur{\cur{d_{i-1}}^{w_0}}"]\arrow[d, phantom, ""{coordinate, name=Z0}] &  &  & \ind{P_{J_{1}}^{w_0}}{G}\cur{\cur{M_{i}}^{w_0}}\arrow[dlll, swap, "\ind{P_{J_{1}}^{w_0}}{G}\cur{\cur{d_{i}}^{w_0}}", rounded corners, to path={ -- ([xshift=2ex]\tikztostart.east)|- (Z0) [near end]\tikztonodes-|([xshift=-2ex]\tikztotarget.west) -- (\tikztotarget)}] \\
        & \ind{P_{J_{1}}^{w_0}}{G}\cur{\cur{M_{i+1}}^{w_0}}\arrow[rrr, "\ind{P_{J_{1}}^{w_0}}{G}\cur{\cur{d_{i+1}}^{w_0}}"] &  &  & \cdots
    \end{tikzcd}
    }.
    \]
    is $\cur{G,P_{J_2}^{w_0}}$-exact and so is
    \[
    \adjustbox{scale=1}{%
    \begin{tikzcd}[row sep=huge]
        \cdots\arrow[r] & \ind{P_{J_{1}}}{G}\cur{M_{i-1}}\arrow[rr, "\ind{P_{J_{1}}}{G}\cur{d_{i-1}}"]\arrow[d, phantom, ""{coordinate, name=Z0}] &  & \ind{P_{J_{1}}}{G}\cur{M_{i}}\arrow[dll, swap, "\ind{P_{J_{1}}}{G}\cur{d_{i}}", rounded corners, to path={ -- ([xshift=2ex]\tikztostart.east)|- (Z0) [near end]\tikztonodes-|([xshift=-2ex]\tikztotarget.west) -- (\tikztotarget)}] &  & \\
        & \ind{P_{J_{1}}}{G}\cur{M_{i+1}}\arrow[rr, "\ind{P_{J_{1}}}{G}\cur{d_{i+1}}"] &  & \cdots
    \end{tikzcd}
    }
    \]
\end{proof}

\Cref{iotres} tells us that $\ind{P_{J_{1}}}{G}$ takes $\cur{P_{J_{1}},P_{J_{1}}\cap P_{J_{2}}^{w_0}}$-exact sequences to $\cur{G,P_{J_2}^{w_0}}$-exact sequences. That then gives us a relative generalized Frobenius reciprocity result for CPS coupled parabolic systems, which we state below.

\begin{theorem}\label{rgfrcpscps}
    Let $G$ be a connected semisimple algebraic group, $B$ a Borel subgroup, $S$ a corresponding set of simple roots, $J_{1}$ and $J_{2}$ proper subsets of $S$ such that $J_{1}^*\cup J_{2}=S$, where $J_{1}^*=-w_0\cur{J_{1}}$ and $w_{0}$ is the long word in the Weyl group $W=N_G(T)/T$, $M$ a $G$-module, and $N$ a $P_{J_{1}}$-module. Then $$\Ext{\cur{G,P_{J_{2}}^{w_0}}}{i}\cur{M,\ind{P_{J_{1}}}{G}(N)}\cong\Ext{\cur{P_{J_{1}},P_{J_{1}}\cap P_{J_{2}}^{w_0}}}{i}\cur{M,N}\text{ for }i\geq0.$$
\end{theorem}

\begin{proof}
    By \Cref{iriri} and \Cref{iotres}, $\ind{P_{J_{1}}}{G}$ takes $\cur{P_{J_{1}},P_{J_{1}}\cap P_{J_{2}}^{w_0}}$-injective modules to $\cur{G,P_{J_{2}}^{w_0}}$-injective modules, and takes $\cur{P_{J_{1}},P_{J_{1}}\cap P_{J_{2}}^{w_0}}$-exact sequences to $\cur{G,P_{J_{2}}^{w_0}}$-exact sequences, so apply \Cref{crrdf2} to get the desired result.
\end{proof}

In the above result, it is important that $J_{1}$ and $J_{2}$ are proper subsets of $S$, so the same argument can not be used when taking one of $J_1$ or $J_2$ to be the empty set, as the condition that $J_{1}^*\cup J_{2}=S$ would force the other to be all of $S$. Despite this, the result will still hold giving a trivial statement in either case. In particular, if $J_1=\emptyset$, then $$\Ext{\cur{G,G^{w_0}}}{i}\cur{M,\ind{B}{G}(N)}\cong\Ext{\cur{B,B}}{i}\cur{M,N},$$ and if $J_2=\emptyset$, then $$\Ext{\cur{G,B^{w_0}}}{i}\cur{M,\ind{G}{G}(N)}\cong\Ext{\cur{G,B^{w_0}}}{i}\cur{M,N}.$$
\subsection{}

We will now show that Ext relative to a parabolic subgroup yields trivial cohomology.

\begin{theorem}\label{rtp}
    Let $G$ be a reductive algebraic group, $B$ a Borel subgroup, $P\geq B$ a parabolic subgroup, and $M$ and $N$ be $G$-modules. Then $$\Ext{\cur{G,P}}{i}\cur{M,N}=0\textup{ for }i>0.$$
\end{theorem}

\begin{proof}
    Consider the $(G,P)$-bar resolution of $N$,
    \[
    \adjustbox{scale=1}{%
    \begin{tikzcd}
        0\arrow[r] & N\arrow[r, "\del_{-1}"] & N\otimes k[G]^{P}\arrow[r, "\del_{0}"] & N\otimes k[G]^{P}\otimes k[G]^{P}\arrow[r, "\del_{1}"] & \cdots.
    \end{tikzcd}
    }
    \]
    $k[G]^{P}\cong\ind{P}{G}(k)\cong k$ (see \cite[II.4.6]{Jan03}), so the $(G,P)$-bar resolution is isomorphic to
    \[
    \adjustbox{scale=1}{%
    \begin{tikzcd}
        0\arrow[r] & N\arrow[r, "\del_{-1}"] & N\arrow[r, "\del_{0}"] & N\arrow[r, "\del_{1}"] & \cdots
    \end{tikzcd}
    }
    \]
    where $\del_{i}=\id_{N}$ if $i$ is odd and $\del_{i}=0$ if $i$ is even (this can be checked using the maps given in \cite[p. 273]{Kim65}). Hence, it is clear that the $(G,P)$-bar resolution of $N$ is split, and so applying $\Hom{G}{M,-}$ to it yields a split exact sequence, meaning that $\Ext{\cur{G,P}}{i}\cur{M,N}=0$ for $i>0$.
\end{proof}

Note that using the same argument as above, one can show that if $\ind{H}{G}(k)=k$, then $$\Ext{\cur{G,H}}{i}\cur{M,N}=0\textup{ for }i>0.$$

Now combining \Cref{rgfrcpscps} and \Cref{rtp}, one immediately obtains the following:

\begin{corollary}
    Let $G$ be a connected semisimple algebraic group, $B$ a Borel subgroup, $S$ a corresponding set of simple roots, $J_{1}$ and $J_{2}$ proper subsets of $S$ such that $J_{1}^*\cup J_{2}=S$, where $J_{1}^*=-w_0\cur{J_{1}}$ and $w_{0}$ is the long word in the Weyl group $W=N_G(T)/T$, $M$ a $G$-module, and $N$ a $P_{J_{1}}$-module. Then $$\Ext{\cur{G,P_{J_{2}}^{w_0}}}{i}\cur{M,\ind{P_{J_{1}}}{G}(N)}\cong\Ext{\cur{P_{J_{1}},P_{J_{1}}\cap P_{J_{2}}^{w_0}}}{i}\cur{M,N}=0\text{ for }i>0.$$
\end{corollary}
\subsection{}


We have now seen that taking $\Ext{}{}$ of a reductive algebraic group relative to either a torus or a parabolic subgroup yields trivial cohomology. So in order to get a useful new invariant, one will need to look relative to a different closed subgroup. The next most likely candidate is a Levi subgroup, which fittingly seems to be the hardest to compute. While computing $\Ext{(G,L)}{i}$ presents a challenge, we are able to say something when $i=1$.

\begin{proposition}
    Let $G$ be a reductive algebraic group, $B$ a Borel subgroup, $P\geq B$ a parabolic subgroup, $L$ a Levi factor of $P$, and $U$ the unipotent subgroup of $P$ such that $P=L\ltimes U$, then for $G$-modules $M$ and $N$, there is a isomorphism
    \[
    \adjustbox{scale=1}{
    \begin{tikzcd}
        \Ext{(G,L)}{1}\cur{M,N}\arrow[r, "\sim"] & \Ext{(P,L)}{1}\cur{M,N}.
    \end{tikzcd}
    }
    \]
\end{proposition}

\begin{proof}

    From \cite[Theorem 2.2]{Kim66}, $\Ext{(G,H)}{1}\cur{M,N}$ is isomorphic to equivalence classes of extensions of $N$ by $M$ that are $(G,H)$-exact. Using this there is a clear restriction map $\Ext{(G,L)}{1}\cur{M,N}\rightarrow\Ext{(P,L)}{1}\cur{M,N}$ which takes a $(G,L)$ short exact sequence to a $(P,L)$ short exact sequence.

    We will first show that this restriction map is injective. Consider two extensions,
    \[
    \adjustbox{scale=1}{%
    \begin{tikzcd}
        0\arrow[r] & N\arrow[r] & A\arrow[r] & M\arrow[r] & 0
    \end{tikzcd}
    }
    \]
    and
    \[
    \adjustbox{scale=1}{%
    \begin{tikzcd}
        0\arrow[r] & N\arrow[r] & A'\arrow[r] & M\arrow[r] & 0,
    \end{tikzcd}
    }
    \]
    such that upon restriction to $P$, they are equivalent. So there exists an isomorphism
    \[
    \adjustbox{scale=1}{%
    \begin{tikzcd}[
        row sep=small,
        ar symbol/.style = {draw=none,"#1" description,sloped},
        isomorphic/.style = {ar symbol={\cong}},
        equals/.style = {ar symbol={=}},
        ]
        0\arrow[r] & \res{P}{G}(N)\arrow[r]\arrow[d,isomorphic] & \res{P}{G}(A)\arrow[r]\arrow[d,isomorphic] & \res{P}{G}(M)\arrow[r]\arrow[d,isomorphic] & 0\\
        0\arrow[r] & \res{P}{G}(N)\arrow[r] & \res{P}{G}(A')\arrow[r] & \res{P}{G}(M)\arrow[r] & 0
    \end{tikzcd}
    }.
    \]
    Now apply induction to see that $$A\cong A\otimes\ind{P}{G}\cur{k}\cong\ind{P}{G}\cur{\res{P}{G}(A)}\cong\ind{P}{G}\cur{\res{P}{G}(A')}\cong A'\otimes\ind{P}{G}\cur{k}\cong A', \text{\cite[II.4.6]{Jan03}.}$$ Hence, the two extensions are equivalent in $\Ext{(G,L)}{1}\cur{M,N}$, and so the restriction map is injective.

    To see that the restriction map is surjective, consider a $(P,L)$ short exact sequence
    \[
    \adjustbox{scale=1}{%
    \begin{tikzcd}
        0\arrow[r] & \res{P}{G}(N)\arrow[r] & A\arrow[r] & \res{P}{G}(M)\arrow[r] & 0.
    \end{tikzcd}
    }
    \]
    We can then apply $\ind{P}{G}$ to get
    \[
    \adjustbox{scale=1}{%
    \begin{tikzcd}
        0\arrow[r] & N\arrow[r] & \ind{P}{G}(A)\arrow[r] & M\arrow[r] & 0.
    \end{tikzcd}
    }
    \]
    Note that this is exact since $\Rind{1}{P}{G}\cur{\res{P}{G}(N)}=N\otimes\Rind{1}{P}{G}\cur{k}\cong0$, \cite[II.4.6]{Jan03}.

    Now apply $\res{P}{G}$, to get the commuting diagram
    \[
    \adjustbox{scale=1}{%
    \begin{tikzcd}
        0\arrow[r] & \res{P}{G}(N)\arrow[r] & A\arrow[r] & \res{P}{G}(M)\arrow[r] & 0\\
        0\arrow[r] & \res{P}{G}(N)\arrow[r]\arrow[u] & \res{P}{G}\cur{\ind{P}{G}(A)}\arrow[r]\arrow[u] & \res{P}{G}(M)\arrow[r]\arrow[u] & 0
    \end{tikzcd}
    },
    \]
    where the outside maps are just the identity, and so the middle map will also be an isomorphism. Therefore,
    \[
    \adjustbox{scale=1}{%
    \begin{tikzcd}
        0\arrow[r] & \res{P}{G}(N)\arrow[r] & \res{P}{G}\cur{\ind{P}{G}(A)}\arrow[r] & \res{P}{G}(M)\arrow[r] & 0
    \end{tikzcd}
    }
    \]
    is $(P,L)$ exact, and so
    \[
    \adjustbox{scale=1}{%
    \begin{tikzcd}
        0\arrow[r] & N\arrow[r] & \ind{P}{G}(A)\arrow[r] & M\arrow[r] & 0
    \end{tikzcd}
    }
    \]
    is $(G,L)$-exact and restricts to the equivalence class of the original $(P,L)$-exact sequence. Hence, the restriction map is surjective, and thus an isomorphism.
\end{proof}

\nocite{BKN10}
\nocite{CPS83}
\nocite{GGNW21}
\nocite{Hoc56}
\nocite{Hoc61}
\nocite{Hoc63}
\nocite{Jan03}
\nocite{Kim65}
\nocite{Kim66}
\nocite{LNW25}
\nocite{Lan02}
\nocite{Wei94}
\printbibliography

\end{document}